%% file: main.tex
\documentclass[11pt]{amsart} 
\usepackage[utf8]{inputenc}
\usepackage[T1]{fontenc}
\usepackage{lmodern}
\usepackage{amsmath, amsthm, amssymb, amsfonts}
\usepackage{latexsym}
\usepackage{amsxtra}
\usepackage[all]{xy}
\usepackage{nicefrac,mathtools}
\usepackage[shortlabels, inline]{enumitem}
\usepackage{microtype}
\usepackage{hyperref}
\hypersetup{
    colorlinks,
    linkcolor={blue},
    citecolor={blue},
    urlcolor={blue}
}
\usepackage[capitalise]{cleveref}
\usepackage[normalem]{ulem}
\usepackage{comment}
\usepackage{xcolor}
\usepackage{tikz-cd}
\usepackage[T1]{fontenc}
\usepackage[utf8]{inputenc}
\usepackage{thmtools}
\usepackage{todonotes}
\usepackage[marginpar=2cm]{geometry}
\numberwithin{equation}{section}
\theoremstyle{plain}
\newtheorem{thm}[equation]{Theorem}
\newtheorem{lem}[equation]{Lemma}

\newtheorem{prop}[equation]{Proposition}

\newtheorem{cor}[equation]{Corollary}

\theoremstyle{definition}

\theoremstyle{remark}\newtheorem{rem}[equation]{Remark}
\theoremstyle{remark}
 \newtheorem*{rem*}{Remark}
\newtheorem*{remarks*}{Remark} \newtheorem{ex}[equation]{Example}

\usepackage{svg}
\pagecolor{white}

\title[Mixed tori in contact surgery diagrams]{Mixed tori in contact surgery diagrams}

\author{Austin Christian} \email{achris66@calpoly.edu}\address{California Polytechnic State University, San Luis Obispo, CA}

\author{Tanushree Shah} \email{tanushrees@cmi.ac.in}\address{Chennai Mathematical Institute, India}

 \thanks{\emph{Subjclass[2020]}: 57K33 }

\begin{document}
\begin{abstract} 
We develop a diagrammatic framework for applying the symplectic JSJ decomposition to exact symplectic fillings of 3-dimensional contact manifolds.  Namely, we apply the symplectic JSJ decomposition to a contact surgery diagram for some $(Y,\zeta)$, producing a finite collection of contact manifolds, also described diagrammatically, whose exact symplectic fillings determine those of $(Y,\zeta)$.  We apply this technique to recover known symplectic filling classifications for certain lens spaces and torus bundles, and also to provide an algorithm for classifying the exact symplectic fillings of a large class of plumbed 3-manifolds.
\end{abstract}

\maketitle
{
\hypersetup{linkcolor={black}}
}

\section{Introduction}\label{sec:intro}

A central problem in contact and symplectic topology is to understand the 
symplectic fillings of a given contact $3$-manifold.  In some cases, progress can be made by cutting the manifold along certain convex tori.  In particular, the first author and Menke give a JSJ-type theorem showing that one can decompose symplectic fillings along \emph{mixed tori}, thereby reducing the study of fillings of a complicated contact manifold to the study of fillings of simpler contact manifolds.
\begin{thm}[{\cite[Theorem 1.1]{christian2018jsj}}]\label{thm:jsj}
Let $(Y,\zeta)$ be a closed, cooriented 3-dimensional contact manifold and let $(W,\omega)$ be an exact symplectic filling of $(Y,\zeta)$.  If there exists a mixed torus $T\subset(Y,\zeta)$ admitting a standard mixed neighborhood $T^2\times[-1,1]$ with slopes $s_{-1}=-1$, $s_0=\infty$, and $s_1$, then there exists a symplectic manifold $(W',\omega')$ such that:
\begin{itemize}
    \item $(W',\omega')$ is an exact symplectic filling of its boundary $(Y',\zeta')$;
    \item $(Y',\zeta')$ is the result of splitting $(Y,\zeta)$ along $T$ with some integral slope $s$ satisfying $0\leq s\leq s_1-1$;
    \item $(W,\omega)$ can be recovered from $(W',\omega')$ by symplectic round 1-handle attachment.
\end{itemize}
\end{thm}
So far, \Cref{thm:jsj} has typically been applied by choosing a smooth manifold $Y$ whose fillable contact structures are classified (or at least reasonably well understood) via the basic slice techniques of Honda (\cite{honda2000classification,honda2000classification2}) and then using this basic slice description of $(Y,\zeta)$ to identify a mixed torus $T\subset(Y,\zeta)$ as well as the contact manifolds which result from splitting along $T$.  See, for instance, \cite{fossati2020topological,christian2021symplectic,etnyre2021symplectic,christian2023some,etnyre2024symplectic}.

In this note we develop a diagrammatic framework for applying \Cref{thm:jsj}, allowing this tool to be applied to contact manifolds which result from Legendrian surgery, but for which a basic slice description may not be readily available.  Namely, our main result will apply \Cref{thm:jsj} to a contact manifold $(Y,\zeta)$ which is obtained by Legendrian surgery on some other contact manifold $(M,\xi)$ along a Legendrian link $L\subset (M,\xi)$ and then describe the ``split open" contact manifolds produced by \Cref{thm:jsj} as other surgeries on $(M,\xi)$.

As an application, we recover the previously known classifications of symplectic fillings for virtually overtwisted lens spaces (\cite{etnyre2021symplectic,christian2023some}) and torus bundles (\cite{christian2021symplectic}), now seen through our diagrammatic approach.  In addition, our diagrammatic approach allows us to generalize those classifications, providing an algorithm which may be applied to the symplectic fillings of any \emph{good plumbed $3$-manifold} whose plumbing graph is sufficiently tame.

Let us now describe the circumstances under which we will identify a mixed torus.  Throughout this discussion (and the remainder of the article), we assume that all Legendrian knots are oriented, so that we may speak of positive and negative stabilizations.

Consider $(M,\xi)$, obtained via surgery along a 0-sphere from a closed, cooriented (but not necessarily connected) 3-dimensional contact manifold $(\tilde{M},\tilde{\xi})$, and denote by $S\subset (M,\xi)$ the 2-dimensional belt sphere of this surgery.  We will call an oriented link $L\subset(M,\xi)$ an \emph{inconsistent chain} if $L$ is either a Legendrian knot which has been stabilized both positively and negatively or has the form
\[
L = \Lambda_1\sqcup \Lambda_2\sqcup\cdots\sqcup \Lambda_{n-1}\sqcup \Lambda_{n},
\]
$n\geq 2$, where
\begin{enumerate}
    \item for $k=1,n$, $\Lambda_k$ is obtained from some other Legendrian knot by stabilizing once with sign $\sigma_k$;

    \item for $2\leq k\leq n-1$, $\Lambda_k$ is a Legendrian meridian of $\Lambda_{k+1}$, and $\Lambda_2$ is also a Legendrian meridian of $\Lambda_1$; namely, these components are max-tb unknots;
    
    \item the unknots $\Lambda_2,\ldots,\Lambda_{n-1}$ are otherwise unlinked\footnote{However, $\Lambda_1$ and $\Lambda_n$ may be linked.} from $L$;
    
    \item with $\sigma_1=\pm 1$ and $\sigma_{n}=\pm 1$ as above,
    \[
    \sigma_1\cdot\ell k(\Lambda_1,\Lambda_2)\cdot\ell k(\Lambda_2,\Lambda_3)\cdots \ell k(\Lambda_{n-1},\Lambda_{n})\cdot \sigma_{n} = -1.
    \]
\end{enumerate}
Finally, we assume that $L$ intersects $S$ exactly twice, with these intersections occurring along a single component of $L$ and canceling algebraically.  See \Cref{fig:main-thm-setup}.

We may similarly define an oriented Legendrian link $L\subset(M,\xi)$ to be an \emph{inconsistent cycle} if the same conditions hold, except that the knots $\Lambda_1$ and $\Lambda_n$ coincide. If $n\geq 3$, so that $L$ is a multi-component link, this simply forces the identity
\[
\ell k(\Lambda_1,\Lambda_2)\cdot\ell k(\Lambda_2,\Lambda_3)\cdots \ell k(\Lambda_{n-1},\Lambda_{n}) = -1.
\]
But if $\Lambda_1=\Lambda_n$ and $n=2$, then we must make sense of the self-linking number $\ell k(\Lambda_1,\Lambda_2)$.  Because $\Lambda_1$ intersects the surgery sphere $S\subset(M,\xi)$ in exactly two algebraically canceling points, there is a naturally-defined link $\Lambda_1^+\sqcup\Lambda_1^-$, disjoint from $S$, with the property that $\Lambda_1^+\#\Lambda_1^-$ is Legendrian isotopic to $\Lambda_1$; see \Cref{fig:pinching} and \Cref{subsec:background:spheres}.  We then interpret $\ell k(\Lambda_1,\Lambda_1)$ to mean $\ell k(\Lambda_1^+,\Lambda_1^-)$.

\begin{figure}
    \centering
    \includegraphics[width=\linewidth]{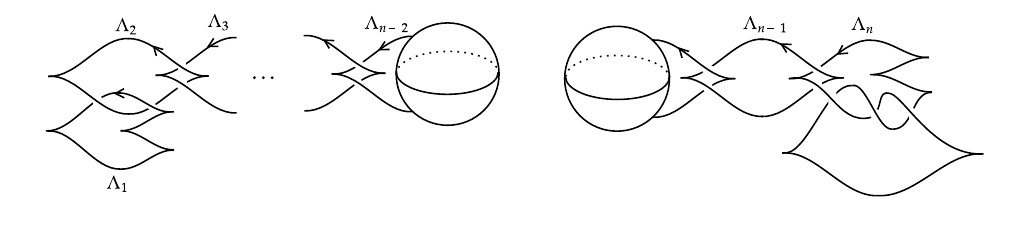}
    \caption{The link $L=\Lambda_1\sqcup\cdots\sqcup\Lambda_n$ is an inconsistent chain which passes over the belt sphere $S$ of a symplectic 1-handle.  The chain is inconsistent because, while both $\Lambda_1$ and $\Lambda_n$ are negatively-stabilized, the product of the linking numbers $\ell k(\Lambda_1,\Lambda_2),\ldots,\ell k(\Lambda_{n-1},\Lambda_n)$ is $-1$.}
    \label{fig:main-thm-setup}
\end{figure}

\begin{rem}
Note that the definitions of inconsistent chain and inconsistent cycle each allow for the possibility that $L$ is a knot, with distinct conditions imposed in the two cases.
\end{rem}

\begin{figure}
    \centering
    \includegraphics[width=\linewidth]{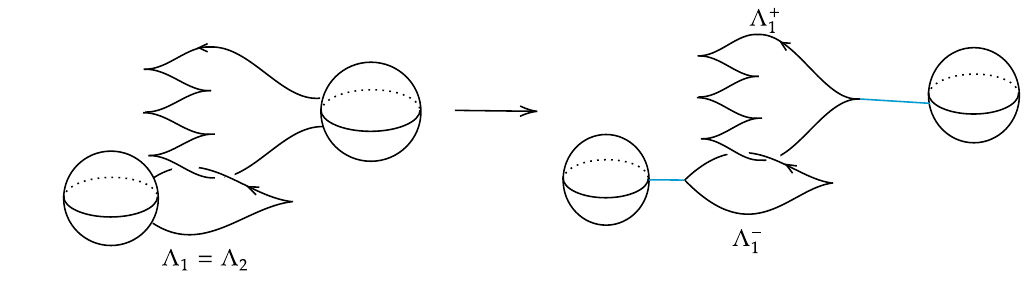}
    \caption{In \Cref{subsec:background:spheres} we describe a canonical splitting of the component $\Lambda_k$ of $L$ which intersects $S$ into a link $\Lambda_k^+\sqcup\Lambda_k^-$.  We use this splitting to define a self-linking number in the case where $\Lambda_1$ and $\Lambda_2$ coincide.}
    \label{fig:pinching}
\end{figure}

We denote by $(Y,\zeta)$ the contact manifold constructed from $(M,\xi)$ via Legendrian surgery along $L$, where $L$ is either an inconsistent chain or an inconsistent cycle. Our main theorem will apply \cref{thm:jsj} to a mixed torus $T$ identified in $(Y,\zeta)$. The result is a list of $n$ contact manifolds obtained by splitting $(Y,\zeta)$ along $T$, and each of these may be expressed as a Legendrian surgery on $(\tilde{M},\tilde{\xi})$ along a link which we now describe.  First, we may isotope $L\subset (M,\xi)$ so that $\Lambda_k$ is its unique component intersecting $S$, for some $1\leq k\leq n$. In \Cref{subsec:background:spheres} we will define a link $\Lambda_k^+\sqcup\Lambda_k^-$ which is disjoint from $S$ and for which $\Lambda_k^+\#\Lambda_k^-$ is Legendrian isotopic, relative to $L\setminus\Lambda_k$, to $\Lambda_k$. The definition of $\Lambda_k^+\sqcup\Lambda_k^-$ is essentially contained in \Cref{fig:pinching}, where we cap off the two components of $\Lambda_k\setminus S$.  We may then interpret $L\setminus\Lambda_k$ as a (possibly empty) link $L_k$ in $(\tilde{M},\tilde{\xi})$ and define $(M_k,\xi_k)$, for $1\leq k\leq n$, to be the result of Legendrian surgery on $(\tilde{M},\tilde{\xi})$ along $L_k$.  Abusing notation, we also have a two-component link $\Lambda_k^+\sqcup\Lambda_k^-$ in $(M_k,\xi_k)$.  At last, we can state our main theorem.

\begin{thm}\label{thm:fillings-of-round-surgery}
Every exact symplectic filling of the contact manifold $(Y,\zeta)$ described above is symplectic deformation equivalent to a filling obtained by attaching a symplectic round 1-handle to an exact symplectic filling of $(M_k,\xi_k)$ along $\Lambda_k^+\sqcup\Lambda_k^-$, for some $1\leq k\leq n$.
\end{thm}

Analogous versions of \Cref{thm:jsj} and \Cref{thm:fillings-of-round-surgery} (and thus all ensuing results) hold for weak symplectic fillings of $(Y,\zeta)$, provided $H_1(Y;\mathbb{R})=0$.

In the special cases where (1) $L$ is a knot; or (2) $(M,\xi)=(S^3,\xi_{std})$ and $L$ is a linear chain of unknots, \cref{thm:fillings-of-round-surgery} can be deduced from \cite{christian2018jsj} or \cite{etnyre2021symplectic,christian2023some}, respectively, as will be observed in \cref{subsec:intro:known-results}. The new contribution of \cref{thm:fillings-of-round-surgery} is to allow $L$ to be any inconsistent chain or inconsistent cycle and $(M,\xi)$ to be any contact manifold.

Given the clunkiness of its setup, \cref{thm:fillings-of-round-surgery} would seem to apply only in a very narrow circumstance.  To demonstrate that \cref{thm:fillings-of-round-surgery} is a generally useful tool, we will first use it to recover some known classifications of symplectic fillings in \cref{subsec:intro:known-results}.  A key observation here is that any contact manifold $(M,\xi)$ may be written as
\[
(M,\xi) \cong (M\#S^3,\xi\#\xi_{\mathrm{std}}),
\]
and thus the assumption that $(M,\xi)$ is obtained via surgery along a 0-sphere is not at all restrictive.  After recovering these known classifications, in \cref{subsec:intro:plumbed} we apply \cref{thm:fillings-of-round-surgery} to a class of contact manifolds whose fillings were previously inaccessible.

\subsection{Recovering some known filling classifications}\label{subsec:intro:known-results}

The first application of \cref{thm:jsj} (here and in \cite{christian2018jsj}) says that Legendrian surgery along a knot which has been stabilized both positively and negatively has no effect on the count of exact symplectic fillings.  In the following, we denote by $S_+(\Lambda)$ and $S_-(\Lambda)$ the positive and negative stabilizations of $\Lambda$, respectively.

\begin{thm}[{\cite[Theorem 1.4]{christian2018jsj}}]\label{thm:mixed-stabilizations}
Suppose $(Y,\zeta)$ is obtained from $(M,\xi)$ via Legendrian surgery along a Legendrian knot $S_+(S_-(\Lambda))$ which has been stabilized both positively and negatively.  Then every exact symplectic filling of $(Y,\zeta)$ is symplectic deformation equivalent to a filling obtained from an exact symplectic filling of $(M,\xi)$ by attaching a Weinstein 2-handle along $S_+(S_-(\Lambda))$.
\end{thm}

An immediate consequence is the following:

\begin{cor}[{\cite[Corollary 1.5]{christian2018jsj}}]\label{cor:mixed-stabilizations-unique-fillings}
If $(Y,\zeta)$ is obtained from $(S^3,\xi_{\mathrm{std}})$ via Legendrian surgery along a Legendrian knot $S_+(S_-(\Lambda))\subset(S^3,\xi_{\mathrm{std}})$, then $(Y,\zeta)$ admits a unique exact filling, up to symplectic deformation equivalence.
\end{cor}

The proof of \cref{thm:mixed-stabilizations} in \cite{christian2018jsj} uses the correspondence between stabilizations and basic slices to recognize $\partial N(S_-(\Lambda))\subset(M,\xi)$ as a torus which becomes mixed after Legendrian surgery along $S_+(S_-(\Lambda))\subset N(S_-(\Lambda))$ and then argues that splitting along this torus returns us to the contact manifold $(M,\xi)$.  Here we present an argument which uses \cref{thm:fillings-of-round-surgery} to identify the mixed torus as the "surgery torus" of a round 1-handle attachment.

\begin{figure}
    \centering
    \includegraphics[width=\linewidth]{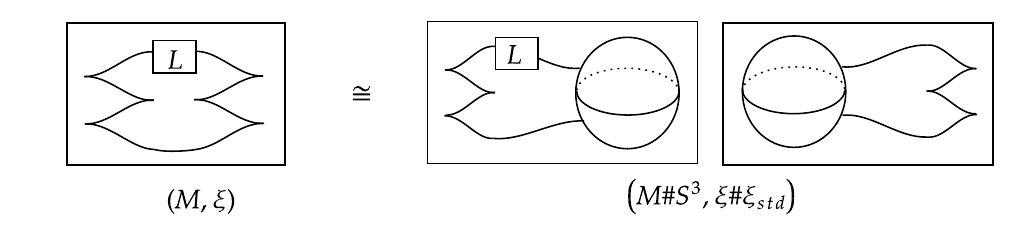}
    \caption{By writing $(M,\xi)$ as the connected sum $(M\# S^3,\xi\#\xi_{\mathrm{std}})$, any Legendrian knot $S_+(S_-(L))$ which has been stabilized both positively and negatively can be realized as an inconsistent chain passing over a surgery sphere.}
    \label{fig:mixed-stabilizations}
\end{figure}

\begin{proof}[Proof of \Cref{thm:mixed-stabilizations}]
As mentioned above, we may realize any contact manifold $(M,\xi)$ as the connected sum $(M\#S^3,\xi\#\xi_{\mathrm{std}})$.  Diagrammatically, this can be carried out by drawing a contact rational Dehn surgery diagram for $(M,\xi)$ next to an empty diagram and connecting the two via a Weinstein 1-handle.  In $(M,\xi)$ we have the Legendrian knot $S_+(S_-(\Lambda))$, and we may push an arc which contains a negative stabilization through the convex sphere which splits the connected sum --- that is, across the 1-handle.  See \Cref{fig:mixed-stabilizations}.  Now by splitting open the connected sum, we obtain the link
\[
S_+(\Lambda) \sqcup S_-(\Lambda_0) \subset (M,\xi) \sqcup (S^3,\xi_{\mathrm{std}}),
\]
where $\Lambda_0$ is the standard Legendrian unknot, as in the splitting depicted in \Cref{fig:pinching}.  \cref{thm:fillings-of-round-surgery} then tells us that every exact symplectic filling of $(Y,\zeta)$ results from attaching a symplectic round 1-handle to an exact symplectic filling of $(M,\xi) \sqcup (S^3,\xi_{\mathrm{std}})$ along $S_+(\Lambda)\sqcup S_-(\Lambda_0)$.

Now $(S^3,\xi_{\mathrm{std}})$ is well-known not to be symplectically co-fillable (c.f. \cite{gromov1985pseudo,Eli01}), so in fact any exact filling of $(M,\xi)\sqcup(S^3,\xi_{\mathrm{std}})$ has the form
\[
(W,\omega)\sqcup(B^4,\omega_{\mathrm{std}}),
\]
where $(W,\omega)$ is an exact filling of $(M,\xi)$.  We attach a symplectic round 1-handle to such a filling along $S_+(L)\sqcup S_-(\Lambda_0)$ by first attaching a Weinstein 1-handle which connects the components $(W,\omega)$ and $(B^4,\omega_{\mathrm{std}})$ into the boundary connected sum
\[
(W,\omega) \natural (B^4,\omega_{\mathrm{std}}) = (W,\omega)
\]
and then attaching a Weinstein 2-handle along
\[
S_+(\Lambda)\# S_-(\Lambda_0) = S_+(S_-(\Lambda)).
\]
This gives the desired description of exact symplectic fillings of $(Y,\zeta)$.
\end{proof}

While \cref{cor:mixed-stabilizations-unique-fillings} gives a uniqueness statement for exact symplectic fillings of certain contact manifolds, note that \cref{thm:mixed-stabilizations} instead gives a \emph{reducing} statement: to classify the exact symplectic fillings of $(Y,\zeta)$, one must first classify the exact symplectic fillings of each $(M_k,\xi_k)$.  Other such statements have since been proven, typically reducing the classification of fillings for virtually overtwisted contact structures to the same problem for universally tight structures.

For instance:
\begin{thm}[{\cite{etnyre2021symplectic,christian2023some}}]\label{thm:lens-space-fillings}
The classification of exact symplectic fillings of virtually overtwisted lens spaces reduces to the classification of exact symplectic fillings of universally tight lens spaces.
\end{thm}

Because Lisca \cite{lisca2008symplectic} has previously classified the exact symplectic fillings of universally tight lens spaces, \cref{thm:lens-space-fillings} completes the classification of fillings for all lens spaces.  The following theorem continues the narrative to a broader class of 3-manifolds.

\begin{thm}[{\cite[Theorem 1.1]{christian2021symplectic}}]\label{thm:torus-bundle-vague}
The classification of exact symplectic fillings of virtually overtwisted hyperbolic torus bundles reduces to the classification of exact symplectic fillings of lens spaces.
\end{thm}

Each of \cref{thm:lens-space-fillings} and \cref{thm:torus-bundle-vague} were originally obtained as applications of \cref{thm:jsj}, taking advantage of the classification of tight contact structures on these manifolds by basic slices. In \cref{sec:known-results-proofs} we will recover each of these theorems as applications of \cref{thm:fillings-of-round-surgery}, without the need for a basic slice decomposition of the full manifold under consideration.

\subsection{Symplectic fillings of plumbed 3-manifolds}\label{subsec:intro:plumbed}

Provided we enforce certain restrictions on the weighted graphs used to construct them, graph manifolds are 3-manifolds which admit many Stein fillable contact structures. Such structures have not been classified in general, and thus one cannot yet hope for results as general as \cref{thm:lens-space-fillings} and \cref{thm:torus-bundle-vague} for graph manifolds. Nonetheless, we can construct many contact structures on graph manifolds for which \cref{thm:fillings-of-round-surgery} simplifies the classification of exact fillings.

Recall that a plumbing graph $\Gamma$ is a weighted graph, and determines a Kirby diagram for a closed, oriented 3-manifold $Y$ as follows: each vertex $v$ corresponds to an unknot in $S^3$, with framing given by the weight $b(v)$, and these unknots are linked according to the edges of $\Gamma$.  In particular, we take the unknots to be oriented and the edges of $\Gamma$ to each be decorated with a $+$ or $-$; the linking number of two unknots whose vertices are connected by an edge is then $\pm 1$ according to the decoration on that edge\footnote{For simply connected graphs, these signs have no effect on the diffeomorphism type of the resulting manifold.}.

With further assumptions on $\Gamma$, each of the unknots described above can be made Legendrian in $(S^3,\xi_{\mathrm{std}})$, with Thurston-Bennequin number $b(v)+1$.  Namely, let us suppose that
\begin{itemize}
    \item $\Gamma$ is either a tree or a graph whose vertices each have degree at most 3;
    \item the weights of $\Gamma$ all satisfy $b(v)\leq -2$.
\end{itemize}
For each vertex $v$, there are $-(b(v)+1)$ Legendrian unknots with Thurston-Bennequin number $b(v)+1$, the rotation numbers of which have the form $b(v)+2k$, $1\leq k\leq -(b(v)+1)$.  We may assign a secondary weight $r(v)$ to each vertex $v$ and use this to determine the rotation number of the corresponding Legendrian unknot.  We will say that a plumbing graph is \emph{fully-decorated} if
\begin{enumerate}
    \item each vertex has a weight $b(v)$ satisfying $b(v)\leq -2$;
    \item each edge has a sign $\pm$;
    \item each vertex has an additional weight $r(v)$ which has the form $b(v)+2k_v$, for some $1\leq k_v\leq -(b(v)+1)$.
\end{enumerate}
Finally, we will say that a plumbing graph is \emph{good} if each vertex $v$ satisfies $b(v)+\deg(v)\leq 0$.  If $\Gamma$ is a good tree, a standard construction (c.f. \cite[Section 4.6]{gompf20234}) produces a Stein handlebody diagram from the data listed here; if $\Gamma$ is instead some good, fully-decorated graph whose vertices all have degree at most 3, then \cite[Section 5]{shah2024tight} provides an algorithm for producing a Stein handlebody diagram.  We review this algorithm in \Cref{sec:plumbed}.

Our primary new application of \Cref{thm:fillings-of-round-surgery} reduces the exact symplectic filling classification problem for such contact manifolds to the same problem for those obtained from \emph{consistent} fully-decorated good plumbing graphs.

\begin{figure}
    \centering
    \includegraphics[width=\linewidth]{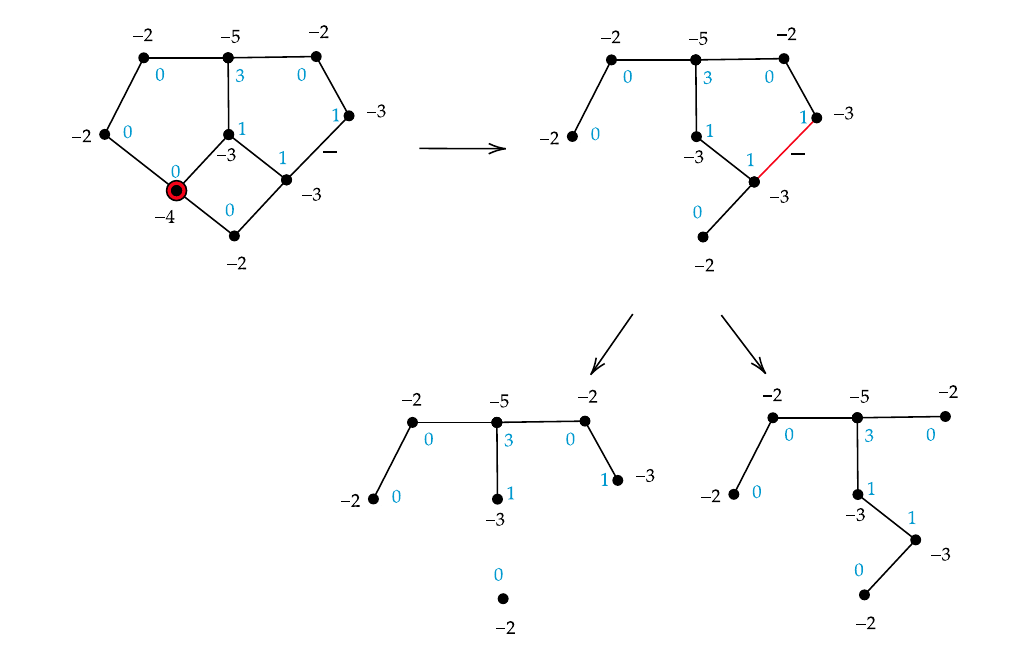}
    \caption{In the upper left is a fully-decorated good plumbing graph $\Gamma$, with undecorated edges assumed to carry a positive sign.  Each vertex $v$ is labeled with $b(v)$ to its left (in black) and $r(v)$ to its right (in blue). Applying \Cref{thm:fillings-of-round-surgery} to $\Gamma$ produces a subgraph.}
    \label{fig:fully-decorated-plumbing}
\end{figure}

\begin{thm}\label{thm:plumbed-3-manifolds}
Let $\Gamma$ be a fully-decorated, good plumbing graph which is either a tree or has no vertices of degree greater than 3, and let $(Y,\zeta)$ be the resulting contact manifold.  Let $\Gamma_1,\ldots,\Gamma_m\subseteq \Gamma$ be the maximal \emph{consistent subgraphs} of $\Gamma$, with resulting contact manifolds $(Y_1,\zeta_1),\ldots,(Y_m,\zeta_m)$.  Then every exact symplectic filling of $(Y,\zeta)$ is obtained by attaching, in a prescribed manner, symplectic round 1-handles to a filling of some $(Y_k,\zeta_k)$, $1\leq k\leq m$.
\end{thm}

Let us note that every symplectically fillable contact structure on a lens space or hyperbolic torus bundle can be depicted via a fully-decorated, good plumbing graph with no vertices of degree greater than 3.  In this way, both \Cref{thm:lens-space-fillings} and \Cref{thm:torus-bundle-vague} can be viewed as corollaries of \Cref{thm:plumbed-3-manifolds}. 

At last, we say that a fully-decorated plumbing graph $\Gamma$ is \emph{consistent} if
\begin{enumerate}
    \item all its secondary weights are extreme, meaning that $r(v)=\pm(b(v)+2)$, for every $v\in\Gamma$;
    \item for every path $e_1,\ldots,e_m$ in $\Gamma$, with corresponding vertices $v_1,\ldots,v_m,v_{m+1}$, the product
    \[
    r(v_1)\cdot \sigma(e_1)\cdots r(v_m)\cdot\sigma(e_m)\cdot r(v_{m+1})
    \]
    is non-negative.
\end{enumerate}
Notice that if $\Gamma$ produces a lens space or hyperbolic torus bundle, then the resulting contact structure is universally tight if and only if $\Gamma$ is consistent.

\begin{ex}
We may use \Cref{thm:plumbed-3-manifolds} to produce an infinite family of contact graph manifolds with unique exact fillings, up to diffeomorphism. Consider the $H$-shaped fully-decorated plumbing graph $\Gamma$ in \Cref{fig:H-shaped}. This graph consists of linear chains $\mathcal{C}_a,\mathcal{C}_b\subset\Gamma$ whose primary weights are $a_1,\ldots,a_m$ and $b_1,\ldots,b_n$, respectively, along with one additional edge connecting $a_i$ and $b_j$, for some $1\leq i\leq m$ and $1\leq j\leq n$, and all edges have positive sign. Let us assume that all primary weights are bounded above by $-5$. All secondary weights are extreme, with those on $\mathcal{C}_a$ being positive and those on $\mathcal{C}_b$ being negative.

Now if $(Y,\zeta)$ is the contact manifold determined by $\Gamma$, then \Cref{thm:plumbed-3-manifolds} tells us that every exact symplectic filling of $(Y,\zeta)$ is obtained by attaching a symplectic round 1-handle to either $(Y_a,\zeta_a)$ or $(Y_b,\zeta_b)$, where $(Y_a,\zeta_a)$ is determined by the graph $\Gamma_a:=\Gamma\setminus\{a_i\}$ and $(Y_b,\zeta_b)$ is determined by the graph $\Gamma_b:=\Gamma\setminus\{b_j\}$, as these are the two maximal consistent subgraphs of $\Gamma$. Now $(Y_a,\zeta_a)$ is the disjoint union of the lens space determined by $\mathcal{C}_b$ with the connected sum of the lens spaces determined by the two components of $\mathcal{C}_a\setminus\{a_i\}$; a similar description holds for $(Y_b,\zeta_b)$, and all relevant lens spaces are universally tight. Moreover, our assumption that the primary weights are bounded above by $-5$ ensures that their constituent lens spaces each have precisely one exact filling up to diffeomorphism; indeed, these are the lens spaces provided by Lisca in the proof of \cite[Corollary 1.2(b)]{lisca2008symplectic}. It follows, using Eliashberg's result in \cite{Eli01} for symplectic fillings of connected sums, that the exact fillings for $(Y_a,\zeta_a)$ and $(Y_b,\zeta_b)$ determined by the plumbing graphs in \Cref{fig:H-shaped} (where the secondary weights have been suppressed) are unique up to diffeomorphism. For either filling, adding a symplectic round 1-handle according to the prescription given by \Cref{thm:plumbed-3-manifolds} yields the exact filling of $(Y,\zeta)$ determined by $\Gamma$; it follows that this exact filling is unique up to diffeomorphism.
\end{ex}

\begin{figure}
    \centering
    \input{H-shaped-graph}
    \caption{The fully-decorated, good plumbing graph $\Gamma$ on the left has two maximal consistent subgraphs.}
    \label{fig:H-shaped}
\end{figure}

\begin{ex}
\Cref{fig:fully-decorated-plumbing} depicts a fully-decorated good plumbing graph $\Gamma$ in its upper left corner, while the bottom row shows the two maximal consistent subgraphs $\Gamma_1,\Gamma_2\subset\Gamma$. No consistent subgraph of $\Gamma$ can include the highlighted vertex $v$, since the highlighted vertex $v$ has secondary weight $r(v)=0\neq\pm(-4+2)$ which is not extreme. The edge decorated with a negative sign must also be excluded from any consistent subgraph, since the vertices at its ends have positive secondary weights. This edge is removed by deleting either of the vertices at its ends, resulting in the two maximal consistent subgraphs of $\Gamma$. Finally, \Cref{thm:plumbed-3-manifolds} tells us that every filling of $(Y,\zeta)$ is obtained by attaching symplectic round 1-handles to a filling of $(Y_1,\zeta_1)$ or $(Y_2,\zeta_2)$, where $(Y,\zeta)$, $(Y_1,\zeta_1)$, and $(Y_2,\zeta_2)$ are the contact manifolds determined by $\Gamma$, $\Gamma_1$, and $\Gamma_2$, respectively.
\end{ex}

\subsection{Organization}
In \Cref{sec:background} we very briefly recall the vocabulary necessary for using \Cref{thm:jsj} and then use \Cref{thm:jsj} to prove \Cref{thm:fillings-of-round-surgery} in \Cref{sec:proof}. In \Cref{sec:known-results-proofs} we prove \Cref{thm:lens-space-fillings} and \Cref{thm:torus-bundle-vague} using \Cref{thm:fillings-of-round-surgery}, then prove \Cref{thm:plumbed-3-manifolds} in \Cref{sec:plumbed}.

\subsection*{Acknowledgments} AC is partially supported by an AMS-Simons Travel Grant and NSF grant DMS-2532551. We thank the referee for their valuable and constructive comments.

\section{Background}\label{sec:background}
We assume familiarity with contact surgery (c.f. \cite{GCT,ozbagci2013surgery}), Kirby and Stein diagrams for contact and symplectic manifolds (c.f. \cite{ozbagci2013surgery,gompf20234}), and basic slice decompositions (c.f. \cite{honda2000classification,honda2000classification2}), and gather here some of the other key tools in our arguments.

\subsection{Splitting along convex spheres}\label{subsec:background:spheres}
The setup of \Cref{thm:fillings-of-round-surgery} considers a contact 3-manifold $(M,\xi)$ obtained from another via surgery along a 0-sphere; here we offer some recollections on this construction.  First, surgery along a 0-sphere $\{q_+,q_-\}\subset(\tilde{M},\tilde{\xi})$ is performed by removing Darboux ball neighborhoods of the points $q_+$ and $q_-$ and identifying the resulting boundary components:
$$M := (M\setminus(N(q_+)\cup N(q_-)))/\sim,$$
where $N(q_+)$ and $N(q_-)$ are open standard neighborhoods of $q_+$ and $q_-$, and $\sim$ identifies the spheres $\partial N(q_+)$ and $\partial N(q_-)$ in a manner which preserves their characteristic foliations.

Conversely, suppose $S\subset(M,\xi)$ is a convex sphere whose dividing set $\Gamma_s$ consists of a single component.  Then we may define
\[
\tilde{M} := D_+ \cup_{\psi_+} \overline{M\setminus S} \cup_{\psi_-} D_-,
\]
where each of $D_+$ and $D_-$ is a Darboux ball and the maps $\psi_\pm\colon\partial D_\pm\to\partial(\overline{M\setminus S})$ preserve characteristic foliations.

We establish this notation primarily so that we may split Legendrian knots which pass through $S$ in a well-defined manner.  Namely, suppose a Legendrian knot $L \subset (M,\xi)$ intersects $S$ in exactly two points $p_+$ and $p_-$ which are algebraically cancelling with respect to $\Gamma_S$.  In this situation, we adopt the convention that splitting along $S$ produces a two-component Legendrian link $L_+\sqcup L_- \subset \tilde{M}$ constructed as
\[
L_+ \sqcup L_- := \gamma_+ \cup_{\psi_+} L \cup_{\psi_-} \gamma_-,
\]
where $\gamma_\pm\subset D_\pm$ is a trivial, unstabilized Legendrian arc in the Darboux ball $D_\pm$ connecting $p_+$ to $p_-$.  Because these arcs are trivial, we may isotope $L_+\sqcup L_-$ to avoid intersection with the Darboux balls $D_\pm\subset\tilde{M}$, and thus also consider $L_+\sqcup L_-$ as a link in $M$.

A fundamental result due to Eliashberg \cite{Eli01} states that convex spheres also govern the behavior of symplectic fillings: if $(W,\omega)$ is a symplectic filling of $(M,\xi)$ and $S\subset M$ is a convex sphere, then $(W,\omega)$ can be obtained by attaching a Weinstein 1-handle to a symplectic filling $(W',\omega')$, with the belt sphere of this 1-handle being $S$.  As a result, decompositions of $(M,\xi)$ along convex spheres naturally extend to the fillings of $(M,\xi)$.

\subsection{Symplectic round $1$-handles}

Here we recall the basic structure of a symplectic round $1$-handle and its effect on the boundary contact manifold.  

A \emph{round $1$-handle} is topologically given by $S^{1} \times D^{1} \times D^{2}$, attached to a $4$-manifold along $S^{1} \times \partial D^{1} \times D^{2}$.  It may be viewed as the product of $S^{1}$ with a standard 3-dimensional $1$-handle.  In the symplectic setting, one endows this piece with a symplectic form compatible with a Liouville structure such that its boundary carries an induced contact structure.  We will not need the details of this construction here, which the reader can find in \cite{adachi2014contact,adachi2017round}.

From the contact perspective, attaching a symplectic round $1$-handle to a symplectic filling of a contact $3$-manifold $(M,\xi)$ has the effect of performing contact surgery along a pair of Legendrian knots $K_{+}$ and $K_{-}$, which appear as the attaching regions of the handle.  More precisely, the operation corresponds to attaching a Weinstein $1$-handle along $S^0=\{p_+,p_-\}$, with $p_{\pm}\in K_{\pm}$, followed by a Weinstein $2$-handle attached along $K_+\#K_+$.

As a result, the boundary contact manifold changes by a \emph{round surgery} along $K_{+} \sqcup K_{-}$.  Contact round surgery along $K_+\sqcup K_-\subset(M,\xi)$ removes standard open neighborhoods $N(K_+)$ and $N(K_-)$ from $M$ and identifies the resulting boundary components in a manner which identifies the meridians (with an orientation-reversal) and dividing sets of $\partial N(\Lambda_+)$ and $\partial N(\Lambda_-)$.  This procedure preserves tightness under suitable conditions and provides a systematic way to construct new contact manifolds from known ones.  

\subsection{Decomposing symplectic fillings along mixed tori}

In this subsection we describe the use of \Cref{thm:jsj}.  Let us suppose that $(Y,\zeta)$ admits a \emph{mixed torus} $T\subset(Y,\zeta)$, meaning that $T$ is convex and admits a neighborhood $T^2\times[-1,1]$ in $Y$ so that
\begin{itemize}
    \item $T=T^2\times\{0\}$;
    \item the restriction of $\xi$ to $T^2\times[-1,1]$ is virtually overtwisted;
    \item each of $T^2\times[-1,0]$ and $T^2\times[0,1]$ is a basic slice.
\end{itemize}
The symplectic JSJ decomposition uses $T$ to remove a symplectic round 1-handle from any exact filling of $(Y,\zeta)$.

Namely, let us recall from \Cref{thm:jsj} that the symplectic JSJ decomposition theorem considers a mixed torus neighborhood $T^2\times[-1,1]\subset(Y,\zeta)$ with slopes $s_{-1}=-1$, $s_0=\infty$, and $s_1$, as well as an exact symplectic filling $(W,\omega)$ of $(Y,\zeta)$, and produces an exact symplectic filling $(W',\omega')$ such that
\begin{enumerate}
    \item $(W,\omega)$ is recovered from $(W',\omega')$ by symplectic round 1-handle attachment;
    \item the boundary of $(W',\omega')$ results from splitting $(Y,\zeta)$ along $T$ with some slope $s$ satisfying $0\leq s\leq s_1-1$.
\end{enumerate}
The slopes $s_i$ refer to the slopes of the dividing curves on the convex tori $T^2\times\{i\}$, $i=-1,0,1$.  By modifying the identification of $T^2$ with $\mathbb{R}^2/\mathbb{Z}^2$, we may always normalize these slopes so that $s_{-1}=-1$ and $s_0=\infty$, as in the statement of the theorem.  With these slopes fixed, we may \emph{split $(Y,\zeta)$ along $T$ with slope $s$}, for any integer $s$, by attaching a solid torus to $Y\setminus T$ along each of its two boundary components:
\[
Y_s := S_{-1} \cup_{\psi_{-1}} (Y\setminus T) \cup_{\psi_1} S_1.
\]
The gluing maps $\psi_i\colon\partial S_i\to T_i$ are chosen so that the meridian of $S_i$ is mapped to a curve of slope $s$ in $T_i$; here $T_{-1}$ and $T_1$ are the components of $\partial(\overline{Y\setminus T})$.
 The contact structure $\zeta_s$ on $Y_s$ agrees with $\zeta$ on $Y\setminus T$, while the solid tori $S_{-1}$ and $S_1$ carry the unique tight contact structures determined by the characteristic foliation along $\partial S_{-1}$ and $\partial S_1$.

 The practical result of \Cref{thm:jsj} is that identifying a mixed torus in a contact manifold $(Y,\zeta)$ allows us to reduce the problem of classifying exact symplectic fillings for $(Y,\zeta)$ to the same problem for the collection of the contact manifolds which result from splitting $(Y,\zeta)$ along $T$ with an allowable slope.

\section{Proof of \cref{thm:fillings-of-round-surgery}}\label{sec:proof}

Let us recall the hypotheses of \cref{thm:fillings-of-round-surgery}.  We have $(M,\xi)$, obtained via surgery along a 0-sphere from a closed, cooriented (but not necessarily connected) 3-dimensional contact manifold $(\tilde{M},\tilde{\xi})$, and denote by $S\subset (M,\xi)$ the 2-dimensional belt sphere of this surgery.  We also have an inconsistent chain or inconsistent cycle $L\subset(M,\xi)$ which intersects $S$ exactly twice, with these intersections occurring along a single component of $L$ and canceling algebraically.

If $L=\Lambda_1\sqcup\cdots\sqcup\Lambda_n$, then for each $1\leq k\leq n$ we have a contact manifold $(M_k,\xi_k)$ obtained from $(\tilde{M},\tilde{\xi})$ via Legendrian surgery along the link $L\setminus\Lambda_k$ and containing a two-component link $\Lambda_k^+\sqcup\Lambda_k^-$.

We prove \cref{thm:fillings-of-round-surgery} as an application of the JSJ decomposition for symplectic fillings, and thus there are two steps:
\begin{enumerate}
    \item In \cref{prop:mixed-belt-torus} we identify a mixed torus $T\subset (Y,\zeta)$, along with a standard mixed neighborhood $T^2\times[-1,1]\subset (Y,\zeta)$ having slopes $s_{-1}=-1$, $s_0=\infty$, and $s_1=n$.
    \item In \cref{prop:splitting-along-belt-torus} we verify that the contact manifold which results from splitting $(Y,\zeta)$ along $T$ with slope $s=k-1$ is precisely $(M_{k},\xi_{k})$, for $k=1,\ldots,n$.
\end{enumerate}

\begin{prop}\label{prop:mixed-belt-torus}
Under the hypotheses of \cref{thm:fillings-of-round-surgery}, $(Y,\zeta)$ contains a mixed torus $T$ which admits a standard mixed neighborhood $T^2\times[-1,1]\subset (Y,\zeta)$ such that $T=T^2\times\{0\}$, $s_{-1}=-1$, $s_0=\infty$, and $s_1=n$.
\end{prop}

\begin{proof}
Performing a Legendrian isotopy if necessary, let us assume that $\Lambda_1$ is the unique component of $L$ intersecting $S\subset(M,\xi)$, so that we may consider the link $\tilde{L}:=L\setminus\Lambda_1$ as being contained in $(\tilde{M},\tilde{\xi})$.  In fact, the intersection of $\Lambda_1$ with $S$ determines the link $\Lambda_1^+\sqcup\Lambda_1^-$, and we may isotope $L$ to ensure that
\begin{itemize}
    \item if $n=1$, then $\Lambda_1^+$ and $\Lambda_1^-$ contain at least one positive and negative stabilization, respectively;
    \item if $n\geq 2$, then $\Lambda_1^-$ is a Legendrian meridian of $\Lambda_2$.
\end{itemize}
Because $\Lambda_1^+\sqcup\Lambda_1^-$ is disjoint from $S$, we may consider this link to be contained in $(\tilde{M},\tilde{\xi})$ and note that $(Y,\zeta)$ is obtained from $(\tilde{M},\tilde{\xi})$ by performing Legendrian surgery along $\tilde{L}$ and round surgery along $\Lambda_1^+\sqcup\Lambda_1^-$.  That is, in addition to Legendrian surgery along $\tilde{L}$, standard neighborhoods of $\Lambda_1^+$ and $\Lambda_1^-$ are removed, with their resulting boundary tori identified according to their meridians and dividing sets.  Our first task, then, is to determine these dividing sets after Legendrian surgery along $\tilde{L}$ has been performed.

First, let us choose a Legendrian knot $\tilde{\Lambda}_1^+$ from which $\Lambda_1^+$ is obtained by stabilization and denote by $\lambda_+$ and $\mu_+$ the contact longitude and meridian, respectively, of $N(\tilde{\Lambda}_1^+)$.  Then the contact longitude of $\Lambda_1^+$ is $\lambda_+-\mu_+$.  If $n=1$, then we may use the same reasoning to choose a solid torus neighborhood with contact longitude and meridian $\lambda_-$ and $\mu_-$ with respect to which the contact longitude of $\Lambda_1^-$ is $\lambda_--\mu_-$.  Notice that the thickened tori $N(\tilde{\Lambda}_1^\pm)\setminus N(\Lambda_1^\pm)$ are basic slices, since their two torus boundary components have dividing sets parallel to $\lambda_\pm$ and $\lambda_\pm-\mu_\pm$, respectively.

If $n\geq 2$, then let us choose a Legendrian knot $\tilde{\Lambda}_n$ from which $\Lambda_n$ is obtained by stabilization.  Observe that $\tilde{L}$ may be assumed to be contained in a neighborhood $N(\tilde{\Lambda}_n)$, since $\Lambda_n=S_{\sigma_2}(\tilde{\Lambda}_n)$ is contained in $N(\tilde{\Lambda}_n)$, $\Lambda_{n-1}$ is a meridian of $\Lambda_n$, and so on.  Using $\lambda_-$ and $\mu_-$ to denote the contact longitude and meridian, respectively, of $N(\tilde{\Lambda}_n)$, we see that $\lambda_--\mu_-$ is the contact longitude of $\Lambda_n$.  Performing Legendrian surgery along $\Lambda_n$ will therefore replace the meridian $\mu_-$ with
\[
\mu_--(\lambda_--\mu_-) = 2\mu_--\lambda_-.
\]
Following this surgery along $\Lambda_n$, \cite[Proposition 2]{ding2009handle} tells us that the meridian $\Lambda_{n-1}$ will become isotopic (as an unoriented Legendrian knot) to the contact framing $\lambda_--\mu_-$.  Legendrian surgery along $\Lambda_{n-1}$ then transforms the meridian into
\[
(2\,\mu_--\lambda_-) - (\lambda_--\mu_-) = 3\,\mu_- - 2\,\lambda_-.
\]
We iterate this process and observe that, following Legendrian surgery along $\Lambda_2\sqcup\cdots\sqcup\Lambda_n$, the meridian $\Lambda_1^-$ has a neighborhood $N(\Lambda_1^-)\subseteq N(\tilde{\Lambda}_n)$ whose boundary has dividing set parallel to $\lambda_--\mu_-$ and whose meridian is
\[
n\,\mu_- - (n-1)\,\lambda_-.
\]
Moreover, because the complement $N(\tilde{\Lambda}_n)\setminus N(\Lambda_1^-)$ is a thickened torus with boundary dividing curves parallel to $\lambda_-$ and $\lambda_--\mu_-$, it is a basic slice.

Round surgery will now delete $N(\Lambda_1^+)$ and $N(\Lambda_1^-)$ and identify their boundaries according to a map $\gamma\colon\partial N(\Lambda_1^+)\to\partial N(\Lambda_1^-)$ which satisfies
\begin{equation}\label{eq:gluing-map}
\gamma(\lambda_+-\mu_+) = \lambda_- - \mu_-
\quad\text{and}\quad
\gamma(\mu_+) = (n-1)\,\lambda_- - n\,\mu_-.
\end{equation}
Notice that $\gamma$ reverses the orientation of the meridians.  An immediate effect of round surgery is that the resulting surgery torus $T$ is now sandwiched between the basic slices $N(\tilde{\Lambda}_1^+)\setminus N(\Lambda_1^+)$ and $N(\tilde{\Lambda}_n)\setminus N(\Lambda_1^-)$.  The signs of these basic slices are determined by the signs of the stabilizations used to construct them, along with the identification of $T$ with $\mathbb{R}^2/\mathbb{Z}^2$, and the definition of inconsistent chains and cycles determines that these signs are opposite.  So the surgery torus $T$ that results from round surgery is mixed, and it remains to compute the slopes of its standard mixed neighborhood.

To this end, notice that
\[
\gamma(\lambda_+) = n\,\lambda_- - (n+1)\,\mu_-,
\]
so the dividing set of $\partial N(\Lambda_1^+)$, which is parallel to $\lambda_+$, becomes parallel to $n\,\lambda_- - (n+1)\,\mu_-$ following the round surgery.  We now have a neighborhood $T^2\times[-1,1]$ of the mixed torus $T=T\times\{0\}$ whose relevant slopes are given by
\[
s_{-1} = -\dfrac{n}{n+1},
\quad
s_0 = -1,
\quad\text{and}\quad
s_1 = \infty.
\]
(The latter two slopes corresponding to $\lambda_--\mu_-$ and $\lambda_-$, respectively.)  By applying the change of coordinates
\[
\left(\begin{matrix}
    1 & 1\\
    n-1 & n
\end{matrix}\right),
\]
where $\mu_-$ is represented by the vector $(1,0)^T$ and $\lambda_-$ by $(0,1)^T$, we obtain the slopes
\[
s_{-1} = -1,
\quad
s_0 = \infty,
\quad\text{and}\quad
s_1 = n,
\]
as desired.
\end{proof}

To complete the proof of \cref{thm:fillings-of-round-surgery} we must determine the effect of applying the JSJ decomposition for symplectic fillings to $(Y,\zeta)$ the mixed torus $T$ identified in \cref{prop:mixed-belt-torus}.  Because we have $s_1=n$ in the standard mixed neighborhood $T^2\times[-1,1]$, \cref{thm:jsj} leads us to consider the result of splitting $(Y,\zeta)$ along $T$ with slope $0\leq s\leq n-1$. 

\begin{prop}\label{prop:splitting-along-belt-torus}
Under the hypotheses of \cref{thm:fillings-of-round-surgery}, splitting $(Y,\zeta)$ along the belt torus $T$ with slope $0\leq s\leq n-1$ yields the contact manifold $(M_{s+1},\xi_{s+1})$.
\end{prop}
\begin{proof}
Let us consider the neighborhood $T^2\times[-1,1]$ of the belt torus $T=T^2\times\{0\}$ constructed in the proof of \cref{prop:mixed-belt-torus}, so that
\[
s_{-1} = -1,
\quad
s_0 = \infty,
\quad\text{and}\quad
s_1 = n.
\]
Recall that we split $(Y,\zeta)$ along $T$ with integral slope $0\leq s\leq n-1$ as follows.  First, we remove $T$ and glue in a solid torus along each of the resulting boundary components:
\[
Y^s := S_{-1} \cup_{\psi^s_{-1}} (Y\setminus T) \cup_{\psi^s_1} S_1,
\]
where
\begin{itemize}
    \item each of $S_{-1}$ and $S_1$ is a solid torus;
    \item the gluing map $\psi_i\colon\partial S_i\to T_i$ carries the meridian of $S_i$ to a curve of slope $s$ in $T_i\subset\partial(Y\setminus T)$, where $T_i$ is the component of $Y\setminus T$ which is parallel to $\mathbb{R}^2\times\{i\}$.
\end{itemize}
With these gluing maps, there is a unique tight contact structure on $S_i$ determined by the characteristic foliation of $T_i$, and this structure is used to extend $\zeta$ on $(Y\setminus T)$ to $\zeta^s$ on $Y^s$.

We now consider the two-step process which begins with $(M_1,\xi_1)$, performs contact round surgery along $\Lambda_1^+\sqcup\Lambda_1^-$ to produce $(Y,\zeta)$, and then splits with integer slope $0\leq s\leq n-1$ along $T$ to produce $(Y^s,\zeta^s)$.  In the notation of the proof of \cref{prop:mixed-belt-torus}, the curves
\[
n\,\lambda_- - (n+1)\,\mu_-
\quad\text{and}\quad
\lambda_--\mu_-
\]
are normalized to have slopes $-1$ and $\infty$, respectively.  It follows that a curve of slope $s$ is
\[
(s+1-n)\,\lambda_- + (n-s)\,\mu_-.
\]
With $\gamma\colon\partial N(\Lambda_1^+)\to\partial N(\Lambda_1^-)$ being the gluing map defined by \Cref{eq:gluing-map}, we have
\[
\gamma^{-1}((s+1-n)\,\lambda_- + (n-s)\,\mu_-) = s\,\lambda_+ - (s+1)\,\mu_+.
\]
As a result, the effect of our two-step process on $(M_1,\xi_1)$ is to
\begin{itemize}
    \item remove the solid torus $N(\Lambda_1^+)$, which has meridional slope $0$, and replace it with a solid torus whose meridional slope is $-\tfrac{s}{s+1}$;
    \item remove the solid torus $N(\Lambda_-^0)$, which has meridional slope $-\tfrac{n-1}{n}$, and replace it with a solid torus whose meridional slope is $-\tfrac{n-1-s}{n-s}$.
\end{itemize}
Notice that this process has no effect when $s=0$.  For $0\leq s\leq n-1$, one may repeat the analysis found in the proof of \Cref{prop:mixed-belt-torus} to determine that the solid tori just described are neighborhoods of $\Lambda_{s+1}^+$ and $\Lambda_{s+1}^-$, respectively, in $(M_{s+1},\xi_{s+1})$.
\end{proof}

\section{Recovering some known filling classifications}\label{sec:known-results-proofs}

In this section we obtain \Cref{thm:lens-space-fillings} and \Cref{thm:torus-bundle-vague} as consequences of \Cref{thm:fillings-of-round-surgery}.  As described in \Cref{sec:intro}, neither of these results is new, but we present new proofs here to demonstrate that \Cref{thm:fillings-of-round-surgery} allows us to apply the symplectic JSJ decomposition without a basic slice description of the contact manifold under consideration. 

\subsection{Virtually overtwisted lens spaces}\label{subsec:known-results:lens-spaces}

Let us begin by recalling the classification obtained by Giroux \cite{giroux2000structures} and Honda \cite{honda2000classification} of tight contact structures on lens spaces.  Given coprime integers $p>q\geq 1$, there is a unique sequence of integers $a_1,\ldots,a_n\geq 2$ such that
\[
-\dfrac{p}{q} = [-a_1,-a_2,\ldots,-a_n] := -a_1-\cfrac{1}{-a_2-\cfrac{1}{\ddots-\cfrac{1}{-a_n}}}.
\]
Given such a sequence of integers, a surgery diagram for $L(p,q)$ is given by a chain $C_{p,q}$ of unknots in which the $k^{\mathrm{th}}$ unknot $U_k$ has framing $-a_k$.  We obtain a Legendrian realization of $C_{p,q}$ by ensuring that each unknot in the chain is a Legendrian unknot with Thurston-Bennequiun number corresponding to its framing; that is, $\mathrm{tb}(U_k)=1-a_k$.  There are $a_k-1$ distinct Legendrian unknots with Thurston-Bennequiun number $1-a_k$, up to isotopy, and thus
\[
(a_1-1)(a_2-1)\cdots(a_n-1)
\]
distinct Legendrian realizations of $C_{p,q}$, up to Legendrian isotopy.  Legendrian surgery on $(S^3,\xi_{\mathrm{std}})$ along each of these yields a tight contact structure on $L(p,q)$, and work of Giroux \cite{giroux2000structures} and Honda \cite{honda2000classification} tells us that these structures are distinct up to isotopy, and that all tight contact structures on $L(p,q)$ are constructed in this manner.

A contact structure on $L(p,q)$ obtained as a quotient of $(S^3,\xi_{\mathrm{std}})$ by the usual action of $\mathbb{Z}/p\mathbb{Z}$ is \emph{universally tight}, since it admits a tight universal cover, while a contact structure $\xi_{\mathrm{vot}}$ on $L(p,q)$ which admits an overtwisted finite cover $(M,\xi_{\mathrm{ot}})\to(L(p,q),\xi_{\mathrm{vot}})$ is called \emph{virtually overtwisted}.  In terms of the classification of tight contact structures described above, we obtain a universally tight structure on $L(p,q)$ by taking each unknot $U_k$ in $C_{p,q}$ to have rotation number $a_k-2$ and obtain another by taking each $U_k$ to have rotation number $2-a_k$.  (These structures are contactomorphic, but are only isotopic if $a_k=2$ for all $1\leq k\leq n$.)  Diagrammatically, these are the diagrams in which all stabilizations --- regardless of component in $C_{p,q}$ --- are on the left, or all stabilizations are on the right.  All other tight contact structures on $L(p,q)$ are virtually overtwisted.

In the case of universally tight contact structures, Lisca \cite{lisca2008symplectic} gives a complete classification of the exact symplectic fillings of $(L(p,q),\xi_{\mathrm{ut}})$.  We will not review the details of this classification here, but its existence means that the following result allows for the classification of exact symplectic fillings for any contact structure on $L(p,q)$.

\begin{prop}\label{prop:inconsistent-subchain}
If the Legendrian surgery link $C_{p,q}$ for a particular tight contact structure $\zeta$ on $L(p,q)$ contains an \emph{inconsistent subchain} $C\subseteq C_{p,q}$, then every exact symplectic filling of $(L(p,q),\zeta)$ is obtained by attaching a symplectic round 1-handle to an exact symplectic filling of a disjoint union of lens spaces whose surgery link results from \emph{breaking} $C$.
\end{prop}

Before proving \cref{prop:inconsistent-subchain}, let us establish the meaning of breaking an inconsistent subchain --- the latter bit of vocabulary borrowed from \cite{etnyre2021symplectic}.  An \emph{inconsistent subchain} is an inconsistent chain (c.f. \cref{sec:intro}) $C$ which is a subchain of $C_{p,q}$.  The result of \emph{breaking} a subchain $C$ of $C_{p,q}\subseteq (S^3,\xi_{\mathrm{std}})$ is a link in $(S^3,\xi_{\mathrm{std}})\sqcup(S^3,\xi_{\mathrm{std}})$ --- namely, we remove a single component of $C$ from $C_{p,q}$ and place the two resulting chains of unknots in distinct copies of $(S^3,\xi_{\mathrm{std}})$.  A subchain of $\ell$ components can therefore be broken in $\ell$ ways.  See \Cref{fig:break-along-subchain}.

\begin{figure}
    \centering
    \includegraphics[width=\linewidth]{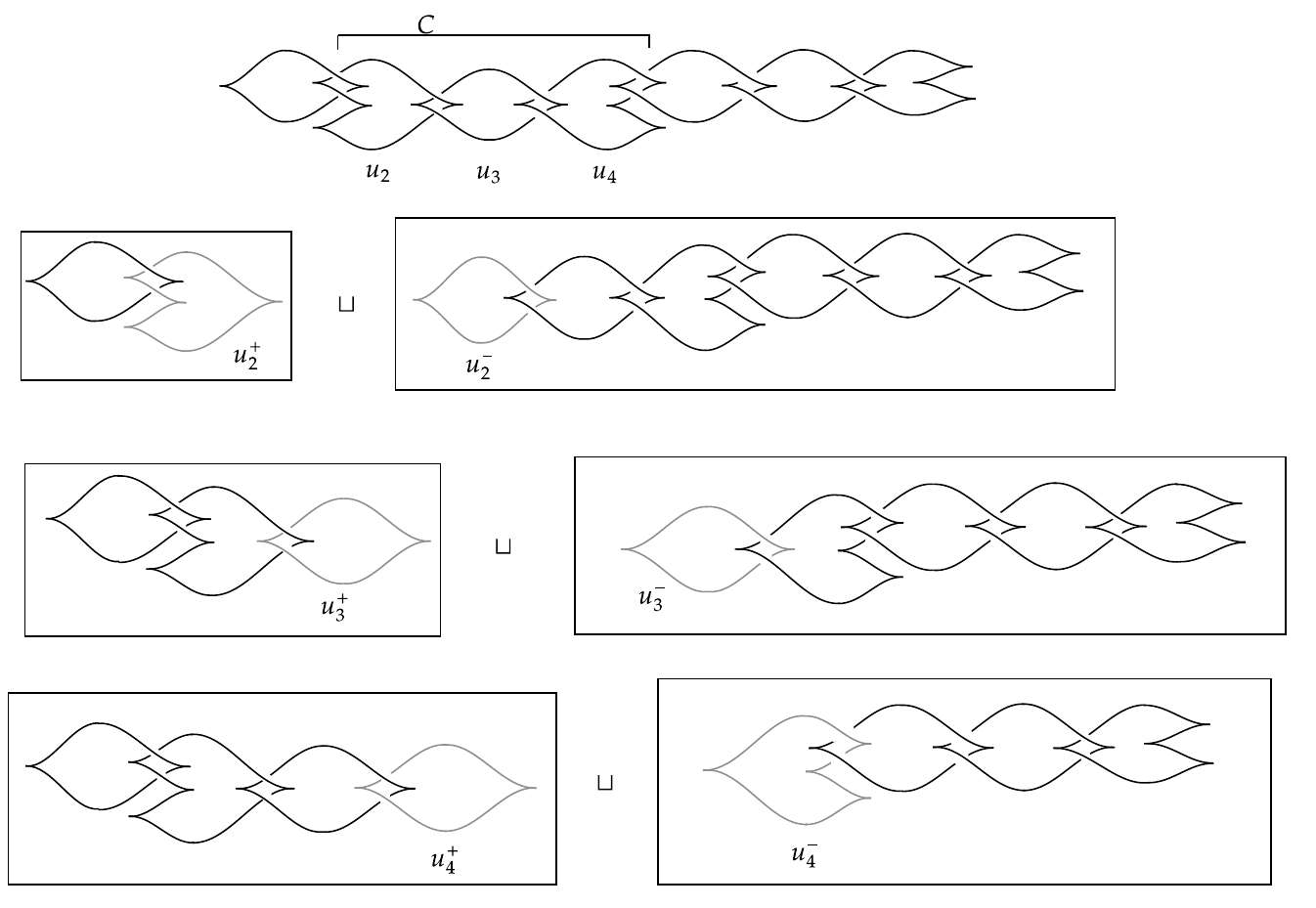}
    \caption{There are $\ell$ ways to break a subchain $C\subseteq C_{p,q}$ of length $\ell$, and each results in a link on which Legendrian surgery produces disjoint union of lens spaces.  In each case, a two-component link $\Lambda_k^+\sqcup\Lambda_k^-$ survives to this disjoint union, depicted here in grey.}
    \label{fig:break-along-subchain}
\end{figure}

\begin{proof}[Proof of \cref{prop:inconsistent-subchain}]
With $C_{p,q}=U_1\sqcup\cdots\sqcup U_n$ as in the statement of \cref{prop:inconsistent-subchain}, let us assume that the inconsistent subchain $C\subseteq C_{p,q}$ consists of the unknots $U_{k_0},\ldots,U_{k_0+\ell-1}$, so that the subchain has length $\ell$ and the stabilized components are $U_{k_0}$ and $U_{k_0+\ell-1}$.  By recognizing $(M,\xi)=(S^3,\xi_{\mathrm{std}})$ as the connected sum
\[
(S^3,\xi_{\mathrm{std}}) = (S^3,\xi_{\mathrm{std}}) \# (S^3,\xi_{\mathrm{std}})
\]
and isotoping $C_{p,q}$ so that one of the components $U_j$, $k_0\leq j\leq k_0+\ell-1$, intersects the resulting surgery sphere $S$, we place ourselves in the setting of \Cref{thm:fillings-of-round-surgery}.  The contact manifolds
\[
(M_1,\xi_1),\ldots,(M_\ell,\xi_\ell)
\]
produced by \Cref{thm:fillings-of-round-surgery} are then precisely the $\ell$ possible results of breaking $C$.
\end{proof}

\cref{thm:lens-space-fillings} follows relatively quickly from \cref{prop:inconsistent-subchain}.  Indeed, the classification of tight contact structures on lens spaces allows us to apply \cref{prop:inconsistent-subchain} to any virtually overtwisted lens space $(L(p,q),\zeta)$, and this reduces the problem of classifying the exact fillings of $(L(p,q),\zeta)$ to the same problem for a collection of disjoint unions of lens spaces whose surgery links have strictly fewer components.  By a combination of results in \cite{schonenberger2005planar} and \cite{etnyre2004symplectic}, no tight contact structure on a lens space is symplectically co-fillable, and thus any exact symplectic filling of a disjoint union of two lens spaces is in fact a disjoint union of fillings of the two lens spaces.  So, after applying \cref{prop:inconsistent-subchain} to an inconsistent subchain $C\subseteq C_{p,q}$ of length $\ell$, we are left to classify the symplectic fillings of (at most) $2\ell$ lens spaces.  Repeated application of \cref{prop:inconsistent-subchain} then reduces the problem to the classification of exact symplectic fillings of some collection of universally tight lens spaces, as described in \cref{thm:lens-space-fillings}.

\subsection{Virtually overtwisted torus bundles}
\cref{prop:inconsistent-subchain} provides an algorithm for constructing all exact symplectic fillings of a virtually overtwisted lens space from the exact symplectic fillings of universally tight lens spaces.  In this subsection we obtain a similar algorithm for virtually overtwisted hyperbolic torus bundles, and in this way prove \cref{thm:torus-bundle-vague}.

We begin by recalling the classification of tight contact structures on hyperbolic torus bundles, established by Honda in \cite{honda2000classification2}.  For any $A\in SL(\mathbb{Z},2)$ we obtain a torus bundle
\[
M_A := (\mathbb{R}^2/\mathbb{Z}^2\times I)/\sim,
\]
where $(\mathbf{x},1)\sim(A\mathbf{x},0)$, and we say that $M_A$ is \emph{hyperbolic} if $|\mathrm{trace}(A)|>2$.  In this case, we may write
\begin{equation} \label{eq:hyperbolic-monodromy}
A = \pm T^{-a_0}ST^{-a_1}S\cdots T^{-a_n}S,
\end{equation}
where $a_0\geq 3$, $a_1,\ldots,a_n\geq 2$,
\[
S = \left(\begin{matrix} 0 & 1 \\ -1 & 0 \end{matrix}\right),
\quad\text{and}\quad
T = \left(\begin{matrix} 1 & 1 \\ 0 & 1 \end{matrix}\right).
\]
We will refer to the monodromy $A$ as \emph{positive} or \emph{negative} according to the sign of \Cref{eq:hyperbolic-monodromy}.

\begin{figure}
    \centering
    \includegraphics[width=\linewidth]{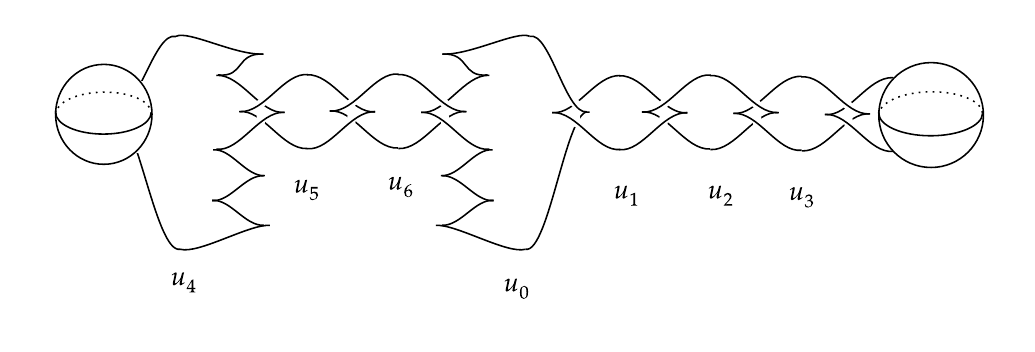}
    \caption{A torus bundle is naturally presented as the result of Legendrian surgery along a chain in $(S^1\times S^2,\xi_{\mathrm{std}})$.  A torus bundle with positive monodromy is presented here; by choosing $U_n$ to have linking number $-1$ with $U_0$ we could obtain a torus bundle with negative monodromy.}
    \label{fig:torus-bundle}
\end{figure}

Honda's classification shows that $M_A$ admits $(a_0-1)(a_1-1)\cdots(a_n-1)$ minimally twisting tight contact structures, up to isotopy, of which
\[
(a_0-1)(a_1-1)\cdots(a_n-1)-(1\pm 1)
\]
are virtually overtwisted, where the sign $\pm$ corresponds to the sign of $A$.  Each of these torus bundles can be described as the boundary of a Weinstein domain as follows.  First, we obtain $(D^3\times S^1,\omega_{\mathrm{std}})$, which has boundary $(S^2\times S^1,\xi_{\mathrm{std}})$, by attaching a Weinstein 1-handle to $(D^4,\omega_{\mathrm{std}})$.  In $(S^2\times S^1,\xi_{\mathrm{std}})$ we have a Legendrian link $L$ which consists of unknots $U_0,U_1,\ldots,U_n$ with the following properties:
\begin{enumerate}
    \item $\mathrm{tb}(U_k)=1-a_k$;
    \item $\mathrm{lk}(U_k,U_{k+1})=1$ for $0\leq k\leq n-1$;
    \item $\mathrm{lk}(U_n,U_0)=1$ if $A$ is positive and $\mathrm{lk}(U_n,U_0)=-1$ if $A$ is negative;
    \item the link $L$ passes over the 1-handle of $(D^3\times S^1,\omega_{\mathrm{std}})$, as depicted in \Cref{fig:torus-bundle}.
\end{enumerate}
These conditions lead us to $(a_0-1)(a_1-1)\cdots(a_n-1)$ isotopy types of tight contact structures on $M_A$, corresponding to the possible isotopy types of $L$.  In all but two of these isotopy types, $L$ will include both a positive stabilization and a negative stabilization; the corresponding contact structures on $M_A$ are virtually overtwisted.  In the two cases where $L$ includes stabilizations\footnote{Note that $L$ always includes at least one stabilization, since $a_0\geq 3$.} of only one sign, $M_A$ is virtually overtwisted if $A$ is negative and universally tight if $A$ is positive.

\Cref{thm:fillings-of-round-surgery} applies to all virtually overtwisted contact structures on hyperbolic torus bundles, but consider first those where, in the description given above, the link $L$ includes at least one stabilization of each sign.  As in \Cref{subsec:known-results:lens-spaces}, this means that we may identify an inconsistent subchain $C\subseteq L$, and so we have the following adaptation of \Cref{prop:inconsistent-subchain} to the present setting.

\begin{prop}\label{prop:torus-bundle-inconsistent-subchain}
If the Legendrian surgery link $L$ for a particular tight contact structure $\zeta$ on $M_A$ contains an inconsistent subchain $C\subseteq L$, then every exact symplectic filling of $(M_A,\zeta)$ is obtained by attaching a symplectic round 1-handle to a connected sum of lens spaces whose surgery link results from breaking $C$.
\end{prop}

Breaking an inconsistent subchain in this setting has a slightly different meaning than in the lens space setting.  If $C\subseteq L\subseteq (S^2\times S^1,\xi_{\mathrm{std}})$ is the inconsistent subchain hypothesized by \Cref{prop:torus-bundle-inconsistent-subchain}, then we break $C$ by arranging for one of its components to be the unique component of $L$ to pass over the 1-handle of $(S^2\times S^1,\xi_{\mathrm{std}})$ and then remove both this component and the 1-handle from the diagram.  The result is a chain of unknots in $(S^3,\xi_{\mathrm{std}})$, and thus a Legendrian surgery diagram for a lens space.  See \Cref{fig:broken-torus-bundle}.  By varying which component of $C$ is arranged to pass over the 1-handle, we find $\ell$ ways of breaking $C$.

\begin{figure}
    \centering
    \includegraphics[width=\linewidth]{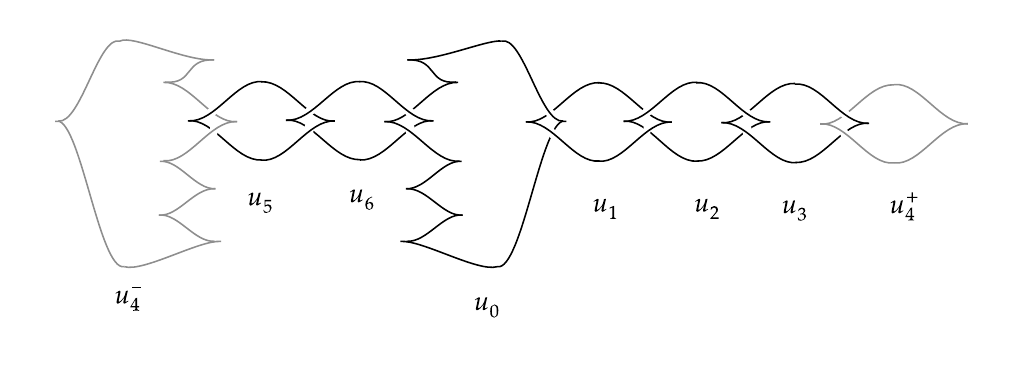}
    \caption{Breaking an inconsistent subchain in a torus bundle produces a lens space which contains a distinguished link $\Lambda_k^+\sqcup\Lambda_k^-$.  Here we present the result of breaking at $U_4$ in the torus bundle depicted in \Cref{fig:torus-bundle}.  Depending on which inconsistent subchain is used, there are either 3 or 4 other breakings which could have been performed.}
    \label{fig:broken-torus-bundle}
\end{figure}

\Cref{prop:torus-bundle-inconsistent-subchain} establishes \Cref{thm:torus-bundle-vague} for all virtually overtwisted hyperbolic torus bundles with positive monodromy, and for all but two such structures in the case of negative monodromy.  It remains to consider the case where the monodromy $A$ is negative and all stabilizations of the Legendrian link $L$ described above are of a single sign.

\begin{prop}
Let $L$ be the surgery link described above for a virtually overtwisted torus bundle $(M_A,\zeta)$, where
\[
A = - T^{-a_0}ST^{-a_1}S\cdots T^{-a_n}S,
\]
for some integers $a_0\geq 3$ and $a_1,\ldots,a_n\geq 2$, and let $m=\max\{k\,|\,a_k\geq 3\}$. If all stabilizations of $L$ are of the same sign, then every exact filling of $(M_A,\zeta)$ is obtained by attaching a symplectic round 1-handle to an exact symplectic filling of $(L(p,q),\xi_{\mathrm{ut}})$, where
\[
-\frac{p}{q}=[-a_{m+1},-a_{m+2},\ldots,-a_n,-a_0,-a_1,\ldots,-a_{m-1}]
\quad\text{or}\quad
-\frac{p}{q}=[-a_1,-a_1,\ldots,-a_{n}]
\]
and $\xi_{\mathrm{ut}}$ is the unique (up to contactomorphism) universally tight contact structure on $L(p,q)$.
\end{prop}

\begin{figure}
    \centering
    \includegraphics[width=0.6\linewidth]{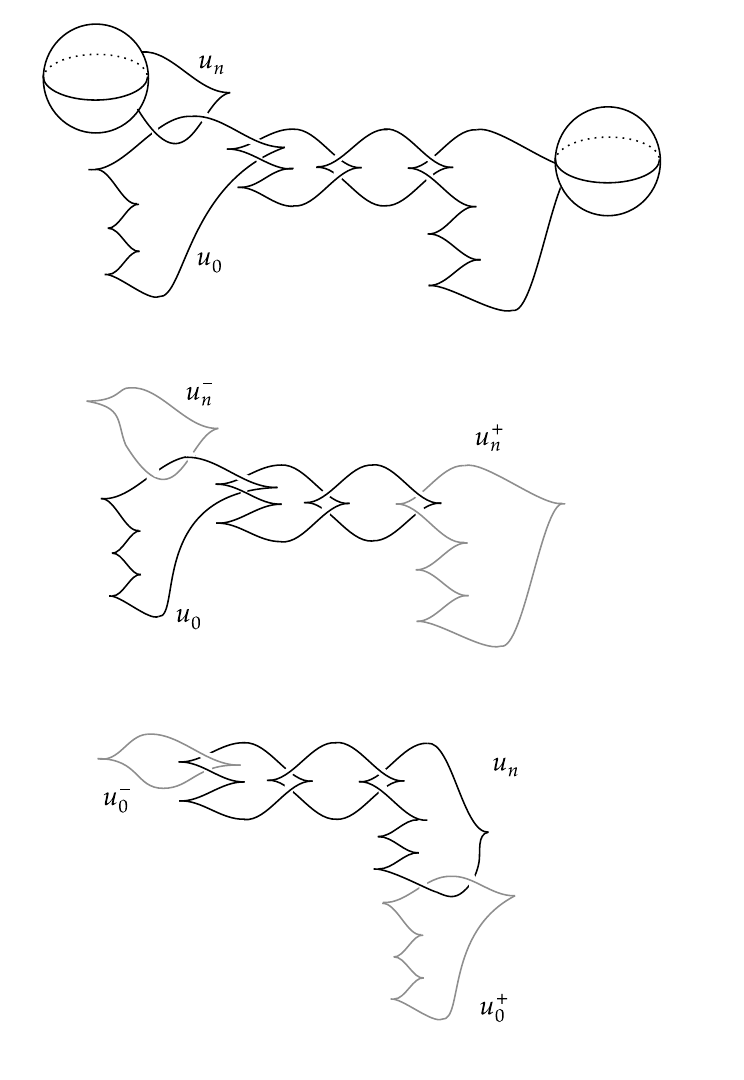}
    \caption{Negative hyperbolic torus bundles admit virtually overtwisted tight contact structures whose Legendrian surgery diagrams include stabilizations of only one sign.  In this case, the linking number $\mathrm{lk}(U_n,U_0)=-1$ provides an inconsistent cycle on which \Cref{thm:fillings-of-round-surgery} may be applied.}
    \label{fig:negative-torus-bundle}
\end{figure}

\begin{proof}
The proof is contained in \Cref{fig:negative-torus-bundle}.  Specifically, the unknot $U_0$ is stabilized at least once and has linking number $-1$ with $U_n$. With $m$ as in the statement of the proposition, the link $U_m\sqcup U_{m+1}\sqcup\cdots\sqcup U_n\sqcup U_0$ forms an inconsistent cycle as hypothesized by \Cref{thm:fillings-of-round-surgery}. Because the unknots $U_{m+1},\ldots,U_n$ are unstabilized, \Cref{thm:fillings-of-round-surgery} provides two ways of breaking this inconsistent cycle, splitting either $U_m$ or $U_0$; the resulting diagrams produce the two lens spaces described by the proposition. See \Cref{fig:negative-torus-bundle} for the case where $m=n$. 
\end{proof}

\section{Symplectic fillings of plumbed 3-manifolds}\label{sec:plumbed}

In this section we prove \Cref{thm:plumbed-3-manifolds} by applying \Cref{thm:fillings-of-round-surgery} to a fully-decorated good plumbing graph $\Gamma$ which is inconsistent. Namely, our application of \Cref{thm:fillings-of-round-surgery} identifies maximal subgraphs $\Gamma_1,\ldots,\Gamma_k$, each obtained by deleting a single vertex from $\Gamma$, such that every exact filling of $(Y,\zeta)$ results from attaching a symplectic round 1-handle to an exact filling of $(Y_j,\zeta_j)$, for some $1\leq j\leq k$. \Cref{thm:plumbed-3-manifolds} then follows by repeated application of this reduction.

\subsection{Wrapped-up graphs and Stein diagrams}
Crucial to our ability to apply \Cref{thm:fillings-of-round-surgery} will be the presentation of $(Y,\zeta)$ as the boundary of a Stein handlebody.  For this we appeal to the techniques of \cite[Section 5]{shah2024tight}, which we briefly recall here.

To describe Stein fillable contact structures on plumbed $3$–manifolds, it is convenient to represent the plumbing graph in a \emph{wrapped-up form}. \Cref{thm:plumbed-3-manifolds} hypothesizes that the plumbing graph $\Gamma$ is either a tree or has no vertices of degree greater than $3$ precisely because \cite[Section 5]{shah2024tight} demonstrates that such graphs may be put into wrapped-up form.

A wrapped-up graph is a planar reformulation\footnote{But a wrapped-up graph need not be planar as a graph.} of the plumbing diagram in which the vertices and edges are arranged so that the linking information between components corresponds directly to Legendrian link diagrams in the boundary of a handlebody. This representation provides a combinatorial bridge between the plumbing description of a $4$–manifold and its Legendrian handle decomposition.
We arrange the plumbing graph $\Gamma$ in the following \emph{wrapped-up form}:
\begin{itemize}
    \item The vertices are organized in horizontal rows.
    \item The vertices in the bottom row lie along a linear subgraph of $\Gamma$, whose endpoints are connected by a curved edge $\gamma$ below the row, forming a cycle $c$ that does not enclose any portion of the graph.
    \item The first and last vertex of each row are incident to curved edges that wrap around the edge $\gamma$.
    \item Every edge of $\Gamma$ is either horizontal, vertical, or a curved edge wrapping around $\gamma$.
    \item Every cycle in $\Gamma$ encloses the innermost cycle $c$.
\end{itemize}
For example, see \Cref{fig:wrapped-up-fully-decorated}.

Placing the graph in this wrapped-up form allows us to apply directly the techniques developed in \cite{gompf20234} to construct a handlebody diagram of the associated $4$-dimensional plumbing. In this representation, each curved edge corresponds to a $1$–handle, each vertex corresponds to a $2$–handle attaching circle, and the linking of the $2$–handles reflects the adjacency relations of the graph. Starting from a fully-decorated, good, wrapped-up graph $\Gamma$, one can construct a \emph{Stein diagram}, which is to say a Kirby diagram whose attaching link is Legendrian, thus describing a Weinstein handlebody. By considering the boundary of this Weinstein handlebody, the Stein diagram also determines a contact manifold. Each vertex $v$ of $\Gamma$ has a corresponding Legendrian unknot in this diagram, with Thurston-Bennequin number $b(v)+1$ and rotation number $r(v)$. That is, the unknot corresponding to $v$ is obtained from the max-$\mathrm{tb}$ unknot by stabilizing $-(b(v)+2)$ times, and the signs of these stabilizations are determined by $r(v)$.  For a fixed system of primary weights on $\Gamma$ --- the weights which determine the diffeomorphism type of the resulting contact manifold $(Y,\zeta)$ --- this construction produces $\prod_{v\in\Gamma}|b(v)+1|$ Stein fillable contact structures, each determined by the secondary system of weights on $\Gamma$.

\begin{figure}
    \centering
    \includegraphics[width=0.8\linewidth]{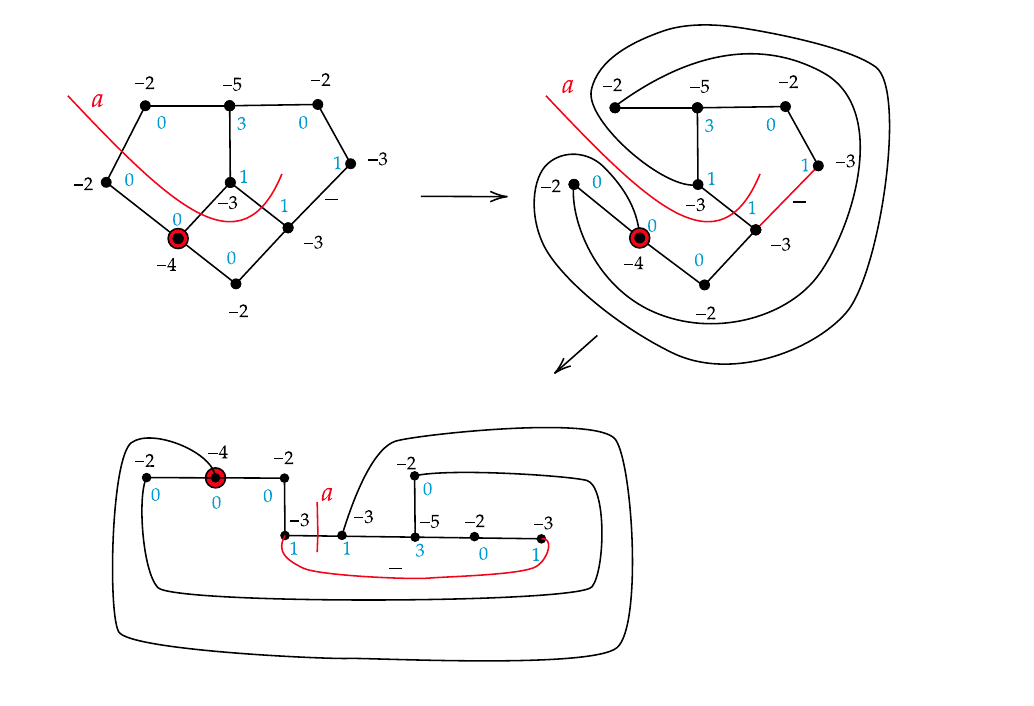}
    \caption{Wrapping up a fully-decorated plumbing graph. Here, the path $a$ is chosen to ensure that the edge decorated with a minus sign will connect the first and last vertices of a row in the wrapped-up form and thus correspond to $1$-handle attachment in the algorithm of \cite[Section 5]{shah2024tight}.}
    \label{fig:wrapped-up-fully-decorated}
\end{figure}
\begin{figure}
    \centering
    \includegraphics[width=0.8\linewidth]{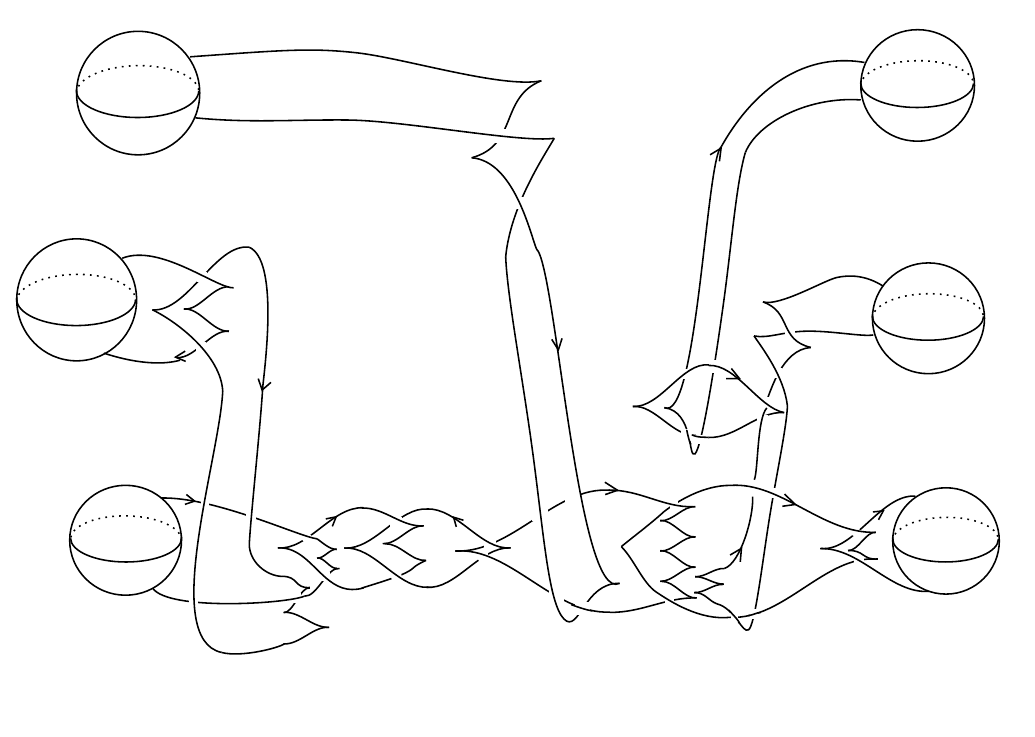}
    \caption{Stein diagram corresponding to the manifold in \Cref{fig:wrapped-up-fully-decorated} using the algorithm of \cite[Section 5]{shah2024tight}.}
    \label{fig: Stein-diagram}
\end{figure}

\begin{lem}\label{lemma:wrapped-up}
Let $\Gamma$ be a fully-decorated, good plumbing graph which is either a tree or has no vertices of degree greater than 3, and let $\mathcal{C}\subset\Gamma$ be a linear chain whose vertices have degree at most 2, with the possible exception of its first and last vertices.  The contact manifold $(Y,\zeta)$ which is determined by $\Gamma$ admits a Stein handlebody diagram in which each vertex of $\Gamma$ corresponds to a Legendrian unknot and the link of unknots corresponding to vertices in $\mathcal{C}$ passes over a 1-handle.
\end{lem}
\begin{proof}
Let us first consider the case where $\Gamma$ is a tree.  In this case, there is a standard construction (c.f. \cite[Section 4.6]{gompf20234}) of a Legendrian link $L\subset(S^3,\xi_{\mathrm{std}})$ consisting of an unknot for each vertex of $\Gamma$, with these unknots linked according to the edges of $\Gamma$.  Legendrian surgery along $L$ produces $(Y,\zeta)$, and thus $L$ can also be interpreted as a Stein diagram for $(Y,\zeta)$.  By introducing a canceling 0-/1-handle pair (or, equivalently, by writing $(S^3,\xi_{\mathrm{std}})$ as $(S^3,\xi_{\mathrm{std}})\#(S^3,\xi_{\mathrm{std}})$), we can describe $(Y,\zeta)$ as the boundary of a Stein domain built from two 0-handles, a single 1-handle, and one 2-handle for each component of $L$.  Moreover, because $\Gamma$ is a tree, we may arrange for $L$ to pass over the 1-handle in the desired manner.  This coincides with our approach to lens spaces depicted, for instance, in \Cref{fig:break-along-subchain}.

If $\Gamma$ is not a tree, then we appeal to the algorithm described in \cite[Section 5]{shah2024tight} for wrapping up graphs whose vertices all have degree no greater than 3.  This algorithm begins with a path $a$ in $\mathbb{R}^2$ which begins in the unbounded region of $\mathbb{R}^2\setminus\Gamma$ and passes through each region of $\mathbb{R}^2\setminus\Gamma$ exactly once; that is, $a$ is a Hamiltonian path in the dual graph of $\Gamma$. We then produce a wrapped-up version of $\Gamma$ by isotoping the edges of $\Gamma$ intersected by $a$ through the point at infinity, and finally applying a planar isotopy to achieve the wrapped-up form described above.  See \cite[Lemma 17]{shah2024tight} for details. As described above, we will obtain a Stein diagram for $(Y,\zeta)$ from this wrapped-up form of $\Gamma$.

It remains to verify that the link of unknots corresponding to $\mathcal{C}$ will pass over a 1-handle in the resulting Stein diagram.  First, if $\mathcal{C}$ is not contained in any cycle, then we may proceed as in the case of a tree, artificially introducing a canceling 0-/1-handle pair.  On the other hand, if $\mathcal{C}$ is contained in a cycle, then $\mathcal{C}$ must be contained in a cycle which forms the boundary of a region $R_0$ in $\mathbb{R}^2\setminus\Gamma$.  This follows from our requirement that the only vertices of $\mathcal{C}$ which may have degree greater than 3 are its extreme vertices. By choosing the path $a$ to terminate at $R_0$, we may use the algorithm of \cite[Section 5]{shah2024tight} to produce a wrapped-up form for $\Gamma$ in which $\mathcal{C}$ passes over a curved edge, and thus the link produced from $\mathcal{C}$ passes over a 1-handle in the resulting Stein diagram.  See \Cref{fig:wrapped-up-fully-decorated} for an example of this algorithm. See \Cref{fig: Stein-diagram} for the Stein diagram corresponding to this manifold.

\end{proof}

\subsection{The proof of \Cref{thm:plumbed-3-manifolds}}
We are now ready to prove \Cref{thm:plumbed-3-manifolds}.  Note that there are two manners in which the fully-decorated good plumbing graph $\Gamma$ may fail to be consistent. The first is that some vertex $v_1$ in $\Gamma$ has a secondary weight $r(v_1)$ which fails to be extreme, as is the case in the upper left graph of \Cref{fig:fully-decorated-plumbing}. In this case, \Cref{lemma:wrapped-up} provides a Stein handlebody diagram for $(Y,\zeta)$ which includes an unknot $U_1$ which has been stabilized at least once with each sign.  We may therefore define $(M,\xi)$ to be the contact manifold obtained using the same Stein handlebody diagram, but with $U_1$ deleted.  That is, $(M,\xi)$ results from the fully-decorated good plumbing graph obtained by deleting $v_1$ from $\Gamma$.  Because $U_1$ may be treated as a Legendrian knot in $(M,\xi)$, $(Y,\zeta)$ satisfies the hypotheses of \Cref{thm:mixed-stabilizations}, and we see that every exact symplectic filling of $(Y,\zeta)$ results from an exact symplectic filling of $(M,\xi)$ by attaching a symplectic round 1-handle.

On the other hand, suppose all secondary weights of $\Gamma$ are extreme.  Then $\Gamma$ is inconsistent if and only if there is a path $e_1,\ldots,e_m$ in $\Gamma$, with corresponding vertices $v_1,\ldots,v_m,v_{m+1}$, such that
\[
r(v_1)\cdot\sigma(e_1)\cdots\sigma(e_m)\cdot r(v_{m+1}) < 0.
\]
Let us call such paths in $\Gamma$ \emph{inconsistent paths} and assume that $e_1,\ldots,e_m$ is a minimal element in the poset of inconsistent paths, partially ordered by inclusion.  Notice that minimality ensures that $e_1,\ldots,e_m$ is a linear chain.

We now claim that each vertex $v_j$, $2\leq j\leq m$, has degree 2.  Indeed, because the secondary weights of $\Gamma$ are extreme, we have $r(v_j)=\pm(b(v_j)+2)$.  If $r(v_j)$ is nonzero for some $2\leq j\leq m$, then we may multiply the above inequality by $r(v_j)^2$ to obtain
\[
r(v_1)\cdot\sigma(e_1)\cdots\sigma(e_{j-1})\cdot r(v_j)\cdot r(v_j)\cdot\sigma(e_j)\cdots\sigma(e_m)\cdot r(v_{m+1}) < 0.
\]
But this would imply that either $e_1,\ldots,e_{j-1}$ or $e_j,\ldots,e_m$ is an inconsistent path, violating the minimality of $e_1,\ldots,e_m$.  So $r(v_j)=0$ and therefore $b(v_j)=-2$.  The fact that $\Gamma$ is good then tells us that $\deg(v_j)\leq -b(v_j)=2$, and the existence of the edges $e_{j-1}$ and $e_j$ implies that $\deg(v_j)=2$.

At last, let us define $\Gamma'$ to be the maximal subgraph of $\Gamma$ obtained by deleting the vertices $v_1,\ldots,v_{m+1}$.  Because $e_1,\ldots,e_m$ is a linear chain, \Cref{lemma:wrapped-up} and its proof produce a Stein diagram $\mathcal{D}$ for $\Gamma$ from which we may obtain a Stein diagram for $\Gamma'$ by deleting the unknots associated to $v_1,\ldots,v_{m+1}$ as well as the 1-handle over which this link passes.  As a result, $(Y,\zeta)$ is of the form hypothesized by \Cref{thm:fillings-of-round-surgery}.  Moreover, \Cref{thm:fillings-of-round-surgery} produces contact manifolds $(M_k,\xi_k)$, $1\leq k\leq m+1$, with a Stein diagram for a filling of $(M_k,\xi_k)$ obtained by deleting from $\mathcal{D}$ the unknot associated with $v_k$, as well as the 1-handle over which this unknot passes.  Note that this is precisely the diagram\footnote{If $\Gamma-\{v_k\}$ is disconnected, then its Stein diagram corresponds to a boundary connected sum.} associated to $\Gamma_k:=\Gamma-\{v_k\}$.  So every exact symplectic filling of $(Y,\zeta)$ results from attaching a symplectic round 1-handle to an exact filling of $(Y_j,\zeta_j)$, for some $1\leq j\leq m+1$, and repeated application of this process reduces us to graphs which are consistent.  This proves \Cref{thm:plumbed-3-manifolds}.

\begin{rem}
In using \Cref{thm:fillings-of-round-surgery} to prove \Cref{thm:plumbed-3-manifolds}, we made a sequence of choices. First, the iterative nature of the proof asks us to identify an ``inconsistency" in the graph $\Gamma$ --- either a vertex with non-extreme secondary weight or a linear chain to which \Cref{lemma:wrapped-up} applies --- and then break $\Gamma$ along this inconsistency, repeating the process until we obtain a maximal consistent subgraph. Distinct sequences of inconsistencies can lead to distinct maximal consistent subgraphs of $\Gamma$; indeed, the goal of \Cref{thm:plumbed-3-manifolds} and its proof is to consider \emph{all} maximal consistent subgraphs. Separately, the proof of \Cref{lemma:wrapped-up} requires the choice of a path which is used to ``wrap up" a fully-decorated plumbing graph. There may be distinct paths satisfying the conditions in that proof, leading to distinct Stein diagrams. But the wrapping-up process does not change the underlying fully-decorated plumbing graph, and thus has no effect on the contact manifolds produced by our algorithm.
\end{rem}

\bibliographystyle{alpha}
\bibliography{references}
\vspace{10pt}
\end{document}

%% file: H-shaped-graph.tex
\begin{tikzpicture}[scale=0.8]
\begin{scope}[shift={(-5,0)}]
\coordinate (a1) at (-4,4);
\coordinate (a2) at (-3.161,3.204);
\coordinate (a3) at (-2.982,2.976);
\coordinate (a4) at (-2.208,1.429);
\coordinate (a5) at (-2.134,1.149);
\coordinate (a6) at (-2,0);
\coordinate (a7) at (-2.134,-1.149);
\coordinate (a8) at (-2.208,-1.429);
\coordinate (a9) at (-2.982,-2.976);
\coordinate (a10) at (-3.161,-3.204);
\coordinate (a11) at (-4,-4);

\coordinate (aTopDot1) at (-2.406,1.973);
\coordinate (aTopDot2) at (-2.528,2.236);
\coordinate (aTopDot3) at (-2.665,2.491);
\coordinate (aBotDot1) at (-2.665,-2.491);
\coordinate (aBotDot2) at (-2.528,-2.236);
\coordinate (aBotDot3) at (-2.406,-1.973);

\coordinate (b1) at (4,4);
\coordinate (b2) at (3.161,3.204);
\coordinate (b3) at (2.982,2.976);
\coordinate (b4) at (2.208,1.429);
\coordinate (b5) at (2.134,1.149);
\coordinate (b6) at (2,0);
\coordinate (b7) at (2.134,-1.149);
\coordinate (b8) at (2.208,-1.429);
\coordinate (b9) at (2.982,-2.976);
\coordinate (b10) at (3.161,-3.204);
\coordinate (b11) at (4,-4);

\coordinate (bTopDot1) at (2.406,1.973);
\coordinate (bTopDot2) at (2.528,2.236);
\coordinate (bTopDot3) at (2.665,2.491);
\coordinate (bBotDot1) at (2.665,-2.491);
\coordinate (bBotDot2) at (2.528,-2.236);
\coordinate (bBotDot3) at (2.406,-1.973);


\draw[thick] (a6) to (b6);

\foreach \l in {a,b}
{
    \draw[thick] (\l 1) to (\l 2);
    \draw[thick] (\l 2) to (\l 3);
    \draw[thick] (\l 4) to (\l 5);
    \draw[thick] (\l 5) to (\l 6);
    \draw[thick] (\l 6) to (\l 7);
    \draw[thick] (\l 7) to (\l 8);
    \draw[thick] (\l 9) to (\l 10);
    \draw[thick] (\l 10) to (\l 11);
    
    \foreach \i in {1,2,3}
    {
        \fill (\l TopDot\i) circle (1.5pt);
        \fill (\l BotDot\i) circle (1.5pt);
    }

    \foreach \i in {1,2,5,6,7,10,11}
    {
        \fill (\l\i) circle (2.5pt);
    }
}

\node at (a1) [left=0.15] {$a_1$};
\node at (a2) [left=0.15] {$a_2$};
\node at (a6) [left=0.15] {$a_i$};
\node at (a11) [left=0.15] {$a_m$};

\node at (a1) [right=0.15,anchor=west,cyan] {$|a_1+2|$};
\node at (a2) [right=0.15,anchor=west,cyan] {$|a_2+2|$};
\node at (a6) [right=0.05,yshift=0.3cm,anchor=west,cyan] {$|a_i+2|$};
\node at (a11) [right=0.15,anchor=west,cyan] {$|a_m+2|$};

\node at (b1) [right=0.15] {$b_1$};
\node at (b2) [right=0.15] {$b_2$};
\node at (b6) [right=0.15] {$b_j$};
\node at (b11) [right=0.15] {$b_n$};

\node at (b1) [left=0.15,anchor=east,cyan] {$-|b_1+2|$};
\node at (b2) [left=0.15,anchor=east,cyan] {$-|b_2+2|$};
\node at (b6) [left=0.05,yshift=-0.3cm,anchor=east,cyan] {$-|b_j+2|$};
\node at (b11) [left=0.15,anchor=east,cyan] {$-|b_n+2|$};
\end{scope}


\begin{scope}[shift={(5,3)},scale=0.5]
\coordinate (a1) at (-4,4);
\coordinate (a2) at (-3.161,3.204);
\coordinate (a3) at (-2.982,2.976);
\coordinate (a4) at (-2.208,1.429);
\coordinate (a5) at (-2.134,1.149);
\coordinate (a6) at (-2,0);
\coordinate (a7) at (-2.134,-1.149);
\coordinate (a8) at (-2.208,-1.429);
\coordinate (a9) at (-2.982,-2.976);
\coordinate (a10) at (-3.161,-3.204);
\coordinate (a11) at (-4,-4);

\coordinate (aTopDot1) at (-2.406,1.973);
\coordinate (aTopDot2) at (-2.528,2.236);
\coordinate (aTopDot3) at (-2.665,2.491);
\coordinate (aBotDot1) at (-2.665,-2.491);
\coordinate (aBotDot2) at (-2.528,-2.236);
\coordinate (aBotDot3) at (-2.406,-1.973);

\coordinate (b1) at (4,4);
\coordinate (b2) at (3.161,3.204);
\coordinate (b3) at (2.982,2.976);
\coordinate (b4) at (2.208,1.429);
\coordinate (b5) at (2.134,1.149);
\coordinate (b6) at (2,0);
\coordinate (b7) at (2.134,-1.149);
\coordinate (b8) at (2.208,-1.429);
\coordinate (b9) at (2.982,-2.976);
\coordinate (b10) at (3.161,-3.204);
\coordinate (b11) at (4,-4);

\coordinate (bTopDot1) at (2.406,1.973);
\coordinate (bTopDot2) at (2.528,2.236);
\coordinate (bTopDot3) at (2.665,2.491);
\coordinate (bBotDot1) at (2.665,-2.491);
\coordinate (bBotDot2) at (2.528,-2.236);
\coordinate (bBotDot3) at (2.406,-1.973);


\foreach \l in {a}
{
    \draw[thick] (\l 1) to (\l 2);
    \draw[thick] (\l 2) to (\l 3);
    \draw[thick] (\l 4) to (\l 5);
    
    \draw[thick] (\l 7) to (\l 8);
    \draw[thick] (\l 9) to (\l 10);
    \draw[thick] (\l 10) to (\l 11);
    
    \foreach \i in {1,2,3}
    {
        \fill (\l TopDot\i) circle (1.5pt);
        \fill (\l BotDot\i) circle (1.5pt);
    }

    \foreach \i in {1,2,5,7,10,11}
    {
        \fill (\l\i) circle (4pt);
    }
}

\foreach \l in {b}
{
    \draw[thick] (\l 1) to (\l 2);
    \draw[thick] (\l 2) to (\l 3);
    \draw[thick] (\l 4) to (\l 5);
    \draw[thick] (\l 5) to (\l 6);
    \draw[thick] (\l 6) to (\l 7);
    \draw[thick] (\l 7) to (\l 8);
    \draw[thick] (\l 9) to (\l 10);
    \draw[thick] (\l 10) to (\l 11);
    
    \foreach \i in {1,2,3}
    {
        \fill (\l TopDot\i) circle (1.5pt);
        \fill (\l BotDot\i) circle (1.5pt);
    }

    \foreach \i in {1,2,5,6,7,10,11}
    {
        \fill (\l\i) circle (4pt);
    }
}

\node at (a1) [left=0.15] {$a_1$};
\node at (a2) [left=0.15] {$a_2$};
\node at (a5) [left=0.15] {$a_{i-1}$};
\node at (a7) [left=0.15] {$a_{i+1}$};
\node at (a11) [left=0.15] {$a_m$};

\node at (b1) [right=0.15] {$b_1$};
\node at (b2) [right=0.15] {$b_2$};
\node at (b6) [right=0.15] {$b_j$};
\node at (b11) [right=0.15] {$b_n$};
\end{scope}


\begin{scope}[shift={(5,-3)},scale=0.5]
\coordinate (a1) at (-4,4);
\coordinate (a2) at (-3.161,3.204);
\coordinate (a3) at (-2.982,2.976);
\coordinate (a4) at (-2.208,1.429);
\coordinate (a5) at (-2.134,1.149);
\coordinate (a6) at (-2,0);
\coordinate (a7) at (-2.134,-1.149);
\coordinate (a8) at (-2.208,-1.429);
\coordinate (a9) at (-2.982,-2.976);
\coordinate (a10) at (-3.161,-3.204);
\coordinate (a11) at (-4,-4);

\coordinate (aTopDot1) at (-2.406,1.973);
\coordinate (aTopDot2) at (-2.528,2.236);
\coordinate (aTopDot3) at (-2.665,2.491);
\coordinate (aBotDot1) at (-2.665,-2.491);
\coordinate (aBotDot2) at (-2.528,-2.236);
\coordinate (aBotDot3) at (-2.406,-1.973);

\coordinate (b1) at (4,4);
\coordinate (b2) at (3.161,3.204);
\coordinate (b3) at (2.982,2.976);
\coordinate (b4) at (2.208,1.429);
\coordinate (b5) at (2.134,1.149);
\coordinate (b6) at (2,0);
\coordinate (b7) at (2.134,-1.149);
\coordinate (b8) at (2.208,-1.429);
\coordinate (b9) at (2.982,-2.976);
\coordinate (b10) at (3.161,-3.204);
\coordinate (b11) at (4,-4);

\coordinate (bTopDot1) at (2.406,1.973);
\coordinate (bTopDot2) at (2.528,2.236);
\coordinate (bTopDot3) at (2.665,2.491);
\coordinate (bBotDot1) at (2.665,-2.491);
\coordinate (bBotDot2) at (2.528,-2.236);
\coordinate (bBotDot3) at (2.406,-1.973);


\foreach \l in {a}
{
    \draw[thick] (\l 1) to (\l 2);
    \draw[thick] (\l 2) to (\l 3);
    \draw[thick] (\l 4) to (\l 5);
    \draw[thick] (\l 5) to (\l 6);
    \draw[thick] (\l 6) to (\l 7);
    \draw[thick] (\l 7) to (\l 8);
    \draw[thick] (\l 9) to (\l 10);
    \draw[thick] (\l 10) to (\l 11);
    
    \foreach \i in {1,2,3}
    {
        \fill (\l TopDot\i) circle (1.5pt);
        \fill (\l BotDot\i) circle (1.5pt);
    }

    \foreach \i in {1,2,5,6,7,10,11}
    {
        \fill (\l\i) circle (4pt);
    }
}
\foreach \l in {b}
{
    \draw[thick] (\l 1) to (\l 2);
    \draw[thick] (\l 2) to (\l 3);
    \draw[thick] (\l 4) to (\l 5);
    
    \draw[thick] (\l 7) to (\l 8);
    \draw[thick] (\l 9) to (\l 10);
    \draw[thick] (\l 10) to (\l 11);
    
    \foreach \i in {1,2,3}
    {
        \fill (\l TopDot\i) circle (1.5pt);
        \fill (\l BotDot\i) circle (1.5pt);
    }

    \foreach \i in {1,2,5,7,10,11}
    {
        \fill (\l\i) circle (4pt);
    }
}

\node at (a1) [left=0.15] {$a_1$};
\node at (a2) [left=0.15] {$a_2$};
\node at (a6) [left=0.15] {$a_i$};
\node at (a11) [left=0.15] {$a_m$};

\node at (b1) [right=0.15] {$b_1$};
\node at (b2) [right=0.15] {$b_2$};
\node at (b5) [right=0.15] {$b_{j-1}$};
\node at (b7) [right=0.15] {$b_{j+1}$};
\node at (b11) [right=0.15] {$b_n$};
\end{scope}

\draw[->] (-1,1.33) to (1,2.66);
\draw[->] (-1,-1.33) to (1,-2.66);
\end{tikzpicture}